\documentclass[11pt]{article}
\usepackage[numbers,sort&compress]{natbib}
\usepackage{enumerate}
\usepackage{amscd}
\usepackage{amsmath}
\usepackage{latexsym}
\usepackage{amsfonts}
\usepackage{amssymb}
\usepackage{amsthm}
\usepackage{verbatim}
\usepackage{mathrsfs}
\usepackage{enumerate}
\usepackage{hyperref}

\theoremstyle{plain}\newtheorem{definition}{Definition}[section]
\theoremstyle{definition}\newtheorem{theorem}{Theorem}[section]
\theoremstyle{plain}\newtheorem{lemma}[theorem]{Lemma}
\theoremstyle{plain}\newtheorem{coro}[theorem]{Corollary}
\theoremstyle{plain}\newtheorem{prop}[theorem]{Proposition}
\theoremstyle{remark}\newtheorem{remark}{Remark}[section]
\usepackage{xcolor}

\newcommand{\Div}{\mathrm{div}\,}
\newcommand{\B}{\Big}
\newcommand{\D}{\text{Div\,}}

\newcommand{\be}{\begin{equation}}
\newcommand{\ee}{\end{equation}}
 \newcommand{\ba}{\begin{aligned}}
 \newcommand{\ea}{\end{aligned}}

\newcommand{\fbxo}{\int_{\tilde{B}_{k}}\!\!\!\!\!\!\!\!\!\! -~\,}
\newcommand{\fbxozero}{\int_{\tilde{B}_{k_{0}}}\!\!\!\!\!\!\!\!\!\!\!\! -~\,}
\newcommand{\fbxozeroo}{\int_{B_{k_{0}}}\!\!\!\!\!\!\!\!\!\!\!\! -~\,}
\newcommand{\fbxoi}{\int_{\tilde{B}_{i}} \!\!\!\!\!\!\!\!\! -~\,}
\newcommand{\fbxoz}{\int_{\tilde{B}_{3}}\! \!\!\!\!\!\!\!\!\! -~\,}

\newcommand{\fbx}{\int_{B_{k}}\!\!\!\!\!\!\!\!\!\! -~}

\newcommand{\kfqRxt}{\iint_{Q_{R}(x,t)} \!\!\!\!\!\!\!\!\!\!\!\!\!\!\!\!\!\!\!\!\!\!\!-\hspace{-0.2cm}-\hspace{-0.3cm}\,~~~~~~}

\newcommand{\fqxoo}{\iint_{\tilde{Q}_{k}} \!\!\!\!\!\!\!\!\!\!\!\!\!\!-\hspace{-0.2cm}-\hspace{-0.3cm}-\,~}

\newcommand{\fqxoth}{\iint_{\tilde{Q}_{3}} \!\!\!\!\!\!\!\!\!\!\!\!\!\!-\hspace{-0.2cm}-\hspace{-0.3cm}-\,~}

\newcommand{\fqxomu}{\iint_{Q(\mu)} \!\!\!\!\!\!\!\!\!\!\!\!\!\!\!\!\!\!-\hspace{-0.2cm}-\hspace{-0.3cm}-\,~}
\newcommand{\fqxol}{\iint_{\tilde{Q}_{l}} \!\!\!\!\!\!\!\!\!\!\!\!\!-\hspace{-0.2cm}-\hspace{-0.3cm}-\,~}
\newcommand{\fqxolo}{\iint_{Q_{l}} \!\!\!\!\!\!\!\!\!\!\!\!\!-\hspace{-0.2cm}-\hspace{-0.3cm}-\,~}

\providecommand{\bysame}{\leavevmode\hbox to3em{\hrulefill}\thinspace}
  \newcommand{\f}{\frac}
    
  \newcommand{\ben}{\begin{enumerate}}
   \newcommand{\een}{\end{enumerate}}

\newcommand{\ti}{\nabla}

\newcommand{\Rmnum}[1]{\expandafter\@slowromancap\romannumeral #1@}

\allowdisplaybreaks

\numberwithin{equation}{section}
\begin{document}
%%%%%%%%%%%%%%%%%%%%%%%%%%%%%%%%%%%%%%%%%%%%%%%%%%%%%%%%%%%%%%%%%%%%%%%%%%%%%%%%%%%%%%%%%%%%%%%%%%%%
\title{Partial regularity of suitable weak solutions to the multi-dimensional generalized magnetohydrodynamics equations}
\author{Wei Ren,\footnote{School of Mathematics and Systems Science, Beihang University, Beijing 100191, P. R. China Email:  renwei4321@163.com}~~\; Yanqing Wang\footnote{Department of Mathematics and Information Science, Zhengzhou University of Light Industry, Zhengzhou, Henan  450002,  P. R. China Email: wangyanqing20056@gmail.com}\; and  Gang Wu\footnote{School of Mathematical Sciences,  University of Chinese Academy of Sciences, Beijing 100049, P. R. China Email: wugangmaths@gmail.com}}
\date{}
\maketitle
\begin{abstract}
In this paper, we are concerned with the partial regularity
of the  suitable weak solutions to the fractional MHD equations in $\mathbb{R}^{n}$ for $n=2,\,3$.
In comparison with the work of
the 3D fractional Navier-Stokes equations
obtained by  Tang and  Yu in
\cite[Commun. Math. Phys. 334: 1455--1482, 2015]{[TY1]}, our results include their   endpoint case $\alpha=3/4$ and the external force belongs to more general parabolic Morrey space. Moreover,
 we prove some interior regularity criteria just via the scaled mixed norm of the  velocity   for the suitable weak solutions to the fractional MHD equations.
 \end{abstract}
%%%%%%%%%%
\section{Introduction}
\label{intro}
\setcounter{section}{1}\setcounter{equation}{0}
We consider  the following generalized incompressible magnetohydrodynamics (MHD)  equations in $\mathbb{R}^{n}$ ($n=2,\,3$)
\be\left\{\ba\label{GMHD}
&u_{t}+ (-\Delta)^{\alpha} u+ u\cdot\ti
u -h\cdot\ti h +\nabla p= f, \\
&h_{t}+(-\Delta)^{\beta} h + u\cdot\ti h-h\cdot\ti u= 0,\\
&\Div u=\Div h=0,\\
&(u,h)|_{t=0}=(u_0,h_0),
\ea\right.\ee
 where $u,\,h$ describe the flow  velocity field and the magnetic field, respectively, the scalar function $p=\pi+\f{1}{2}h^{2}$ stands for the total pressure and the external force  is denoted by  $f$	
with $\Div f=0$.
The fractional Laplacian $(-\Delta)^{\alpha}$ as the infinitesimal generator of a L\'{e}vy process is defined by $\widehat{(-\Delta)^{\alpha}f}(\xi)=|\xi|^{2\alpha}\hat{f}(\xi)$
,   where $\hat{f}(\xi)=\frac{1}{(2\pi)^{n}}\int_{\mathbb{R}^{n}}f (x)e^{-i\xi\cdot x}\,dx$.
 The  initial data $(u_{0},h_{0})$ satisfies $\Div u_{0}=\Div h_{0}=0$.

When $\alpha=\beta=1$, the system \eqref{GMHD} reduces to the classical MHD equations. MHD equations play an important role in electrically conducting fluids such as plasmas(see, e.g.\cite{[Biskamp]}).
There have been extensive
studies
on various topics concerning the MHD system and 
 fractional MHD equations (see, e.g., \cite{[TYZ],[CWY],[JZ],[HX1],[JW1],[PPS],[HX2],[Wu1],[Wu2],[Wu3],[CMZ],[DL],[RW],[Hasegawa],[WZ]} and  references therein).
 MHD equations without the the magnetic field degenerate to the Navier-Stokes equations. It is well-known that both the Navier-Stokes system and the Euler system in $\mathbb{R}^{2}$ are globally
well-posed.  Sermange and   Teman \cite{[ST]} showed that the weak solutions of the 2D MHD system are regular. Very recently, the 2D generalized MHD system has been mathematically investigated in
several works (see, e.g., \cite{[TYZ],[CWY],[JZ],[Wu3]}). However, due to strong coupling between the magnetic field and the velocity field, to our knowledge, whether smooth solutions of the 2D   MHD equations with fractional power dissipation $\alpha=\beta<1$ break down in a finite time
remains   open. In  \cite{[Wu3]},   the global smooth solutions   of  2D generalized MHD  with $\beta<1$ and $\alpha>1$ were established
by  Wu. One goal of this paper is to prove partial regularity of solutions satisfying local energy inequality to the 2D fractional MHD system for  $1/2<\alpha=\beta<1$.

The global weak solutions and the local strong solutions  to the 3D MHD equations were constructed by
 Duvaut  and   Lions \cite{[DL]}, and  Sermange and Teman \cite{[ST]}. Regularity  criteria of weak solutions to the 3D MHD equations only in terms of velocity field were proved in \cite{[HX1],[CMZ]}. Partial regularity of suitable weak solutions to the 3D MHD equations was investigated by  He and  Xin in \cite{[HX2]} (see also \cite{[JW1]}). The interior regularity criteria are shown for the suitable weak solutions via   the velocity field
with sufficiently small local scaled norm and the  magnetic field    with bounded local scaled norm in \cite{[HX2]}.
Very recently,  Wang and  Zhang \cite{[WZ]}  removed    the  magnetic field   hypothesis for the    regularity criteria for the suitable weak solutions to the 3D MHD equations.
These results indicate that the velocity field
  plays a more dominant role than the magnetic field  on the regularity of solutions to the
magnetohydrodynamic equations, which is  consistent with the numerical simulations in \cite{[Hasegawa],[PPS]}.

The partial regularity of suitable weak solutions to the 3D MHD equations obtained in \cite{[HX2]} is an analogue of the celebrated Caffarelli-Kohn-Nirenberg theorem to the 3D Navier-Stokes equations, namely, one dimensional Hausdorff measure of the set of
the possible space-time singular points of suitable weak solutions to the system is zero.  The partial regularity of weak solutions obeying the local energy inequality to the 3D Navier-Stokes equations was
originated from Scheffer \cite{[Scheffer1],[Scheffer2],[Scheffer3]}. The optimal
Hausdorff  dimension estimate of
the possible space-time singular points set of suitable weak solutions to the 3D Navier-Stokes was obtained by  Caffarelli,  Kohn and  Nirenberg in \cite{[CKN]}.
Since then, there have been extensive
studies on the partial regularity    of   solutions to the  Navier-Stokes equations, MHD equations and the related models  with fractional dissipation (see, e.g., \cite{[Lin],[WW],[HX2],[TX],[KP],[JW2],[GKT],[WZ],[TY2],[TY3],[JW1],[TY1],[RW]}).
In particular,  Katz and  Pavlovi\'c \cite{[KP]} proved that
the Hausdorff dimension of the singular set for generalized Navier-Stokes equations
with $1<\alpha<5/4$ at the first blow-up time
is at most  $5-4\alpha$, which was extended to the
generalized MHD equations in \cite{[RW]}.  It is shown in \cite{[JW2]} that the $(5-4\alpha)/2\alpha$ dimensional Hausdorff measure of
  possible time singular  points of  weak solutions to the 3D fractional Navier-Stokes equations on the interval $(0,\infty)$
is zero if $5/6\leq\alpha<5/4$.
Very recently,  Tang and  Yu \cite{[TY3]}
showed that the solutions of the 3D stationary fractional Navier-Stokes equations are regular away from a compact set whose $(5-6\alpha)$-Hausdoff measure is zero in the case
$1/2 <\alpha< 5/6$.

 Based on
 Caffarelli and  Silvestre's extension for the fractional Laplacian operator,
 Tang and  Yu \cite{[TY1]} successfully established the partial regularity of suitable weak solutions to the fractional
Navier-Stokes equations in the case $ 3/4<\alpha<1$ in \cite{[TY1]}, where the
Hausdorff dimension of the potential space-time singular points set of  suitable weak solutions
  is at most $5-4\alpha$.
Since Lemma 2.6 in \cite{[TY1]} collapses when $\alpha=3/4$, it seems that
   the limiting case  $\alpha=3/4$ can not be covered in their work.
One objective of this work is to address this borderline case. This is partially motivated by the previous investigation of partial regularity to the solutions of   the 4D Navier-Stokes equations in \cite{[WW]}.
The following   observation plays an important role in our proof.
  Just as the 4D Navier-Stokes equations, from the interpolation inequality and  Sobolev embedding theorem,
 we find that it holds $u\in L_{t,x}^{3} $ and
$p\in L^{3/2}_{t,x}$ for
the suitable  weak solutions  of the 3D generalized Navier-Stokes equations for $\alpha=3/4$,
 which  ensures that every term in the local energy inequality \eqref{loc} makes sense. Meanwhile, this    means a recurrence relation that the left hand of local energy inequality can control the right hand of local energy inequality.
Specifically,
we devote oneself to treating the partial regularity of suitable weak solutions of the fractional  magnetohydrodynamic equations \eqref{GMHD}.

Throughout this paper,
$v^{\ast}$   denotes the extension  of $v$ associated with the  fractional Laplacian operator $(-\Delta)^{\alpha}$ in the sense \cite{[CS]}, for more details see Sections 2. The norm of parabolic Morrey space $ M_{2\alpha,\gamma}$ will be defined at the end of this section. In what follows, we consider the system \eqref{GMHD} in the case $n/4\leq\alpha=\beta<1(n=2,\,3)$ and $\alpha\neq1/2$ unless otherwise stated.
We are now ready to state the main theorems of this paper.
\begin{theorem}\label{the1.1}
Suppose that the triplet $(u,\,h,\,p)$ is a suitable weak solution to (\ref{GMHD}) and    $f\in L_{t,x}^{q}$ with $q>\f{n+2\alpha}{2\alpha}$ if $1/2<\alpha<3/4$;
$f\in M_{2\alpha,\gamma}$ with $\gamma>0$ if $3/4\leq\alpha<1$
 in $\mathbb{R}^{n}$ with $n=2,\,3$.   Then $u$ and $h$ can be bounded by $1$ on $[-\f{1}{8^{2\alpha}},0]\times B(\f{1}{8})$  provided the following condition holds,
\be\label{cond}\ba \sup_{t\in {[-1,0]}}\int_{B(1)}(|u|^{2}+&|h|^{2})
+\iint_{Q^{\ast}(1)}  y^{1-2\alpha}\big(
|\nabla^{\ast}u^{\ast}|^{2}+|\nabla^{\ast}h^{\ast}|^{2}\big)+
\B(\iint_{Q(1)}|p|^{3/2}\B)^{2/3} +\|f\|  \leq \varepsilon_{1}, \ea\ee
for an absolute constant $\varepsilon_{1}>0$,
where
 $$\|f\|=\left\{\ba
& \|f\|_{L^{q}(Q(1))},\,~~~1/2<\alpha<3/4,\\
&\|f\|_{M_{2\alpha,\gamma}},~~~~~~~3/4\leq\alpha<1.
\ea\right.
$$
\end{theorem}
\begin{remark}
  A slightly different version of Theorem \ref{the1.1} was obtained for the 3D Navier-Stokes equations by Vasseur in \cite{[Vasseur]} and the 4D Navier-Stokes equations in \cite{[WW]} (see \cite{[JW1]} for the MHD equations), where all the proofs  rely on the De Giorigi iteration. Here, we mainly follow the   pathway of \cite{[CKN],[TY1]}  to
prove this theorem.  In contrast with the work of \cite{[CKN],[TY1]}, we will use estimate on pressure $p$ in $L^{3/2}_{t,x} $  norm instead of $L^{p}_{t}L^{1}_{x}$($p>1$) norm  utilized there.
Moreover,   Theorem \ref{the1.1} without the magnetic field seems to be  a new regularity
criterion for the suitable weak solutions of the    fractional Navier-Stokes equations in $\mathbb{R}^{3}$,
which is of independent
interest.
\end{remark}
\begin{theorem}\label{the1.2}
 Assume that
$(u,\,h,\,p)$ is a suitable weak solution to (\ref{GMHD}),  then
 $(0,\,0)$ is a regular point of $u(x,t)$ and $h(x,t)$  if the following condition holds,
\be\label{parabolic}
\limsup_{r\to 0_+}\f{1}{r^{n+2-4\alpha}}\iint_{Q^{\ast}(r)}  y^{1-2\alpha}\big(
|\nabla^{\ast}u^{\ast}|^{2}+  |\nabla^{\ast}h^{\ast}|^{2}\big) \leq\varepsilon_{2},~~~~n=2,~ 3,\ee
for an universal constant $\varepsilon_{2}>0$.
\end{theorem}
\begin{remark}
In $\mathbb{R}^{3}$,
 Vitali  covering lemma utilized in \cite{[CKN]}  together with  Theorem \ref{the1.2} implies that the $5-4\alpha$ ($3/4\leq\alpha<1$) dimensional parabolic Hausdorff measure of the possible  singular points set of $u$ and $h$ is zero for any suitable solution of (\ref{GMHD}), which   extends the recent work of  Tang and  Yu \cite{[TY1]} in the case $3/4<\alpha<1$.
 It should be point out,
just as the 4D Navier-Stokes equations,  that it is not known whether the suitable weak solution to the system (\ref{GMHD}) exists.
At least, we can obtain the partial regularity of smooth solutions of the 3D generalized  Navier-Stokes equations with $\alpha=3/4$ at the first blow-up time.
\end{remark}
\begin{remark}
It worth noting that  Tang and  Yu \cite{[TY1]} proved \eqref{parabolic}
without the magnetic field in $\mathbb{R}^{3}$
under the condition  that the force  $f$ lies in $ L_{t,x}^{q'}$ with $q'>\f{9+6\alpha}{4\alpha+1}$.
Notice  that  $M_{2\alpha,\gamma}\supset L^{q}$ for $q>\f{2\alpha+n}{2\alpha}$ when $0<\gamma\leq2\alpha$ and $M_{2\alpha,\gamma}=\{0\}$ if $\gamma>2\alpha$. Hence, the assumption on the force $f$ in \cite{[TY1]} is relaxed.
Furthermore, under the definition of regular point, the  hypothesis that $f\in L^{q}, q>\f{n+2\alpha}{2\alpha}$ is optical in the sense of   scaling.
We mention that
 the  Caffarelli-Kohn-Nirenberg theorem to the 3D Navier-Stokes equations with the external force belonging to parabolic Morrey space is due to the work of  Ladyzenskaja and    Seregin in \cite{[LS]}, where the proof relies on a blow-up procedure and compact argument.
\end{remark}
Since the existence of
  the magnetic field in MHD equations \eqref{GMHD},
  the proof of Theorem \ref{the1.2}  is  more involved than  the generalized Navier-Stokes equations. Particularly,
a difficulty arises when we deal with the case $\alpha=\beta\leq3/4$. As mentioned above,  Lemma 2.6 in \cite{[TY1]} breaks down in this case.
To  build an effective iteration scheme via local energy inequality \eqref{loc},  our observation is that, under the hypothesis \eqref{parabolic}, the right hand side of the local energy inequality \eqref{loc}
should be seen as the  magnitude  like $\|u\|^{2}_{L^{3}_{t,x}}$ ($\|h\|^{2}_{L_{t,x}^{3}}$) rather than
$\|u\|^{3}_{L^{3}_{t,x}}$ as usual.
Based on this,
we find that Lemma \ref{ineq} established in the next section   instead of Lemma 2.6 with $\alpha>3/4$ in \cite{[TY1]} works for $\alpha>1/2$.
However, this  makes the
  estimate of term  $\iint u\cdot\nabla \phi p$ be subtle.
To this end, making  full use of
 the interior estimate of Harmonic function, we could
establish the decay estimate of pressure $p-\bar{p}$
in  $L_{t,x}^{\f{n+2\alpha}{n}}$ norm. Meanwhile,
the divergence-free algebraical structure  of \eqref{GMHD}
plays a crucial role in
  dealing with the interaction terms between the  magnitude field and the velocity field in the local energy inequality   to avoid the appearance of terms   similar to $\|u\|^{3}_{L^{3}_{t,x}}$.
This enables us
to achieve the proof of Theorem \ref{the1.2}.

Partially motivated by the works \cite{[HX2],[HX1],[CMZ],[WZ]}, we show
 that the velocity field  plays  a more
important role than the  magnetic field in the local regularity
theory of the MHD equations \eqref{GMHD}. Precisely, we shall prove some
interior regularity criteria which do  not
explicitly involve the magnetic field  for
the suitable weak solutions to the MHD equations \eqref{GMHD}. Without loss of
generality, we assume that $f= 0$ in the following theorems.
\begin{theorem}\label{the1.3}
Assume that the triplet
$(u,\,h,\,p)$ is a suitable weak solution  to  (\ref{GMHD}),  then
 $(0,0)$ is regular point of $(u,\,h)$  provided the
following condition holds, \\
  \begin{equation}
\limsup_{r\to 0_+}\,\, \f{1}{r^{n+3-4\alpha}}\iint_{Q(r)}
|u|^{3}dxdt \leq \varepsilon_{3},
\label{33}
\end{equation}
for an universal constant $\varepsilon_{3}>0$.
 \end{theorem}
\begin{remark}
It is worth remarking that the regularity criterion \eqref{33} is given just in terms of the velocity field $u$ instead of the combination between the velocity field and the magnetic field. Moreover, even for the 3D fractional Navier-Stokes system,  although the extension of $u$ appears in  the righthand side of its local energy inequality,  this sufficient regularity condition \eqref{33}  does not involve  $u^{\ast}$.
\end{remark}
The key issue to prove \eqref{33} is to  resort the appropriate test function to  circumvent the straightforward control of
the terms involving the magnetic field $h$ and the extensions $u^{\ast},\,h^{\ast}$ on the right
hand side of the local energy inequality \eqref{loc}.
Furthermore,    the  magnitude of the left  hand of the local energy inequality likes $\|h\|^{2/3}_{L_{t,x}^{3}}$ ($\|p\|^{2/3}_{L_{t,x}^{3/2}}$) helps us to handle  the terms $\iint |u||h|^{2}+|u||p-p_{\rho}|$. This allows us to complete the
proof of Theorem \ref{the1.3}. Due to the  pressure in terms of the
the magnetic field and the velocity field,  it seems difficult to  extend the integral norms with different exponents in space and
time in \eqref{33}. Using the completely different  iteration scheme involving  the pressure and
 a slight variant of  treating
 the terms $\iint |u||h|^{2}+|u||p-p_{\rho}|$ in the proof of Theorem \ref{the1.3} gives
 the regularity conditions via  the velocity with sufficiently small local scaled norm and the magnetic field with bounded local scaled norm.
However, we would like to point out that we can   remove the hypothesis
of the  magnetic field when $(n+4)/8<\alpha=\beta<1$.
\begin{theorem}\label{the1.4}
Assume that the triplet
$(u,\,h,\,p)$ is a suitable weak solution  to  (\ref{GMHD}),  then
 $(0,0)$ is regular point of $(u,\,h)$  provided one of the
following conditions holds.

(1)  For any
constant $M>0$, there exists a positive constant $\varepsilon_{4}(M)$ such that
\begin{align}
\limsup_{r\to 0_+}\,\, r^{-( \frac n{\ell} + \frac {2\alpha}{q} -(2\alpha-1))}
\Big(\int^{0}_{-r^{2\alpha}}\Big(\int_{B(r)}
|u|^{\ell}dx\Big)^{\f{q}{\ell }}ds\Big)^{\f{1}{q}} \leq \varepsilon_{4},\nonumber\\
\limsup_{r\to 0_+}\,\, r^{-( \frac n{\ell} + \frac {2\alpha}{q} -(2\alpha-1))}
\Big(\int^{0}_{-r^{2\alpha}}\Big(\int_{B(r)}
|h|^{\ell}dx\Big)^{\f{q}{\ell }}ds\Big)^{\f{1}{q}} \leq M,\label{hbd}
\end{align}
where the pair $(q,\,\ell)$  satisfies
\be  \label{pq1.1}
 2\alpha- 1\leq \f{n}{\ell} +\f{2\alpha}{q} \leq 2\alpha,~~ 1\leq q\leq\infty,~n=2,\,3.
 \ee

(2)  There exists a positive constant $\varepsilon_{5} $
such that
\begin{equation}
\limsup_{r\to 0_+}\,\, r^{-( \frac n{\ell} + \frac {2\alpha}{q} -(2\alpha-1))}
\Big(\int^{0}_{-r^{2\alpha}}\Big(\int_{B(r)}
|u|^{\ell}dx\Big)^{\f{q}{\ell }}ds\Big)^{\f{1}{q}} \leq \varepsilon_{5},~~
 (n+4)/8<\alpha=\beta<1,
\label{pq1}
\end{equation}
where the pair $(\ell,\,q)$ satisfies
\begin{equation} \label{pq1.2}
 2\alpha- 1\leq \f{n}{\ell} +\f{2\alpha}{q} \leq 2\alpha, \quad ~~~\max\big\{1,\f{n}{4\alpha-n}\big\}<\ell,     ~n=2,\,3.
\end{equation}
\end{theorem}
 \begin{remark}
 As a straightforward consequence of \eqref{hbd}, Serrin type   sufficient regular condition for the 3D generalized Navier-Stokes equations  is obtained. More precisely, let $(u,\,p)$ be the suitable weak solutions to the 3D fractional Navier-Stokes system for $3/4\leq\alpha<1$, then $u$ is regular on $Q(r/2)$ provided that   $u$ lies in $L^{q,\,\ell}(Q(r))$ with $2\alpha/q+3/\ell=2\alpha-1$\,($\ell>3/(2\alpha-1)$ or $\|u\|_{L^{\infty, 3/(2\alpha-1)}(Q(r))}$ is sufficiently small.
\end{remark}
\begin{remark}
We emphasize that the magnetic field with bounded local scaled norm in \eqref{hbd} is only used for treating the term $\iint |u||p-p_{\rho}|$.
\end{remark}

The remainder of this paper is organized as follows.
In the next section, we  will begin with some facts related to Caffarelli and Silvestre's generalized extension and collecting some useful inequalities associated with this extension.
Then we will  present  the definition of the suitable weak solutions and  establish various   dimensionless decay estimates.
In Sections 3, by means of induction argument, we complete the proof of Theorem \ref{the1.1}.
Combining Theorem \ref{the1.1} proved in
Sections 3 with the preliminary lemmas  in Section 2, we prove Theorem  \ref{the1.2} as well as Theorem \ref{the1.3} and Theorem \ref{the1.4} in  Sections 4. Finally, for completeness, an appendix is dedicated to proving inequalities  stated in Section 2.

\noindent
{\bf Notations:} Throughout this paper, we denote
\begin{align*}
     &B(x,\mu):=\{y\in \mathbb{R}^{n}||x-y|\leq \mu\}, && B(\mu):= B(0,\mu), && \tilde{B}(\mu):=B(x_{0},\,\mu),\\
     &B^{\ast}(x,\mu):=B(x,\mu)\times(0,\mu), && B^{\ast}(\mu):= B^{\ast}(0,\mu), && \tilde{B}^{\ast}(\mu):=B^{\ast}(x_{0},\,\mu),\\
     &Q(x,t,\mu):=B(x,\,\mu)\times(t-\mu^{2\alpha}, t),  && Q(\mu):= Q(0,0,\mu), && \tilde{Q}(\mu):= Q(x_{0},t_{0},\mu),\\
     &Q^{\ast}(x,t,\mu):=B^{\ast}(x,\mu)\times(t-\mu^{2\alpha},t), && Q^{\ast}(\mu):= Q^{\ast}(0,0,\mu), && \tilde{Q}^{\ast}(\mu):= Q^{\ast}(x_{0},t_{0},\mu),\\
     &r_{k}=2^{-k},\quad \tilde{B}_{k}:= \tilde{B}(r_{k}), \quad \tilde{B}^{\ast}_{k}:=\tilde{B}^{\ast}(r_{k}), && \tilde{Q}_{k}:=\tilde{Q}(r_{k}), &&\tilde{Q}^{\ast}_{k}:=\tilde{Q}^{\ast}(r_{k}).
\end{align*}
The classical Sobolev norm $\|\cdot\|_{H^{s}}$  is defined as   $\|f\|^{2} _{{H}^{s}}= \int_{\mathbb{R}^{n}} (1+|\xi|)^{2s}|\hat{f}(\xi)|^{2}d\xi$, $s\in \mathbb{R}$.
  We denote by  $ \dot{H}^{s}$ homogenous Sobolev spaces with the norm $\|f\|^{2} _{\dot{H}^{s}}= \int_{\mathbb{R}^{n}} |\xi|^{2s}|\hat{f}(\xi)|^{2}d\xi$.
 For $q\in [1,\,\infty]$, the notation $L^{q}(0,\,T;\,X)$ stands for the set of measurable functions on the interval $(0,\,T)$ with values in $X$ and $\|f(t,\cdot)\|_{X}$ belongs to $L^{q}(0,\,T)$.
   For simplicity,   we write
$$\|f\| _{L^{q,\,\ell}(Q(\mu))}:=\|f\| _{L^{q}(-\mu^{2\alpha},\,0;\,L^{\ell}(B(\mu)))}~~~\text{ and}~~~~
  \|f\| _{L^{q}(Q(\mu))}:=\|f\| _{L^{q,\,q}(Q(\mu))}. $$
  The parabolic Morrey space $M_{2\alpha,\gamma}$ is equipped with the norm
$$
\|f\|_{M_{2\alpha,\gamma}}=
\sup_{(x,\,t)\in\mathbb{R}^{n}\times(-T,0)}\sup_{R>0}
\B\{\f{1}{R^{\gamma-2\alpha}}
\B(\kfqRxt|f|^{2\alpha}\B)^{1/2\alpha}\B\}.
$$
 Denote
  the average of $f$ on the ball $B(r)$ by
$\overline{f}_{r}$ and the average of $p$ on the ball $\tilde{B}(l)$ by
$\overline{p}_{l}$. $K$  stands for the standard normalized fundamental solution of Laplace equation in $\mathbb{R}^{n}$ with $n\geq2$.
We denote by  $\D$ the divergence operator in $\mathbb{R}^{n+1}_{+}$ and  $\nabla^{\ast}$ the gradient operator in $\mathbb{R}^{n+1}_{+}$.
  $|\Omega|$ represents the Lebesgue measure of the set $\Omega$. We will use the summation convention on repeated indices.
 $C$ is an absolute constant which may be different from line to line unless otherwise stated in this paper.

 \section{ Preliminaries}\label{sec2}
\setcounter{section}{2}\setcounter{equation}{0}

In this section, we first recall Caffarelli and Silvestre's generalized extension for the fractional Laplacian operator $(-\Delta)^{s}$ with $0<s<1$ in \cite{[CS]}. The   fractional Laplacian can be interpreted as
\be\label{csextension}
(-\Delta)^{s}u=-C_{s}\lim\limits
_{y\rightarrow0_{+}}y^{1-2s}\partial_{y}u^{\ast},
\ee
where $C_{s}$ is a constant depending only on $s$ and $u^{\ast}$ satisfies
\be\label{ellipticprblom}\left\{\ba
  &\D(y^{1-2s}\nabla^{\ast} u^{\ast})=0 ~~\text{in}~~ \mathbb{R}^{n+1}_{+},\\  &u^{\ast}|_{y=0}=u.\ea\right.
\ee
Furthermore, the Poisson   formula below is valid
\be\label{PF}u^{\ast}(x,y)=C_{n,s}
\int_{\mathbb{R}^{n}}\f{y^{2s}u(\xi)}{(|x-\xi|^{2}+y^{2})^{\f{n+2s}{2}}}
d\xi,\ee
where $C_{n,s}$ is a constant depending only on $n$ and $s$.
In addition, from  Section 3.2 in \cite{[CS]}, an equivalent definition of the $\dot{H}^{s}$ norm reads
\be\label{eqnorm}\|u\|^{2} _{\dot{H}^{s}}= \int_{\mathbb{R}^{n}} |\xi|^{2s}|\hat{u}(\xi)|^{2}d\xi=\int_{\mathbb{R}_{+}^{n+1}} y^{1-2s} | \nabla^{\ast} u^{\ast} |^{2}dxdy.
\ee
As a by-product of \eqref{ellipticprblom}, for any $v |_{y=0}=u$, it holds
\be\label{mini}
\int_{\mathbb{R}_{+}^{n+1}} y^{1-2s} | \nabla^{\ast} u^{\ast} |^{2}dxdy\leq \int_{\mathbb{R}_{+}^{n+1}} y^{1-2s} | \nabla^{\ast}v  |^{2}dxdy.
\ee
With  \eqref{eqnorm} and \eqref{mini} in hand, one can prove the following inequalities  frequently used later:
\begin{align}
&\|u-\overline{u}_{\mu}\|_{L^{\f{2n}{n-2s}}(B(\mu/2))}\leq C\B(\int_{B^{\ast}(\mu)}y^{1-2s}|\nabla^{\ast} u^{\ast}|^{2}\Big)^{1/2},\label{keyinq1}\\
&\|u\|_{L^{\f{2n}{n-2s}}(B(\mu/2))}\leq C\B(\int_{B^{\ast}(\mu)}y^{1-2s}|\nabla^{\ast} u^{\ast}|^{2}\Big)^{1/2}+C\mu^{-s}
\B(\int_{B(\mu)}|u|^{2}\B)^{1/2},\label{keyinq2}
\\
&\|u\|_{L^{2,\,\f{2n}{n-2s}}(Q(\mu/2))}\leq C\B(\iint_{Q^{\ast}(\mu)}y^{1-2s}|\nabla^{\ast} u^{\ast}|^{2}\Big)^{1/2}+C
\B(\sup_{-\mu^{2s}\leq t<0}\int_{B(\mu)}|u|^{2}\B)^{1/2},\label{keyinq7}
\\
&\int_{B^{\ast}(\mu)}y^{1-2s}|u^{\ast}|^{2}\leq C\mu^{2-2s}
\int_{B(\mu)}|u|^{2}+C\mu^{2}\int_{B^{\ast}(\mu)}
y^{1-2s}|\nabla^{\ast} u^{\ast}|^{2},\label{keyinq3}
\\
&\iint_{Q^{\ast}(\mu)}y^{1-2s}|u^{\ast}|^{2}\leq C\mu^{2-2s}
\iint_{Q(\mu)}|u|^{2}+C\mu^{2}\iint_{Q^{\ast}(\mu)}
y^{1-2s}|\nabla^{\ast} u^{\ast}|^{2},\label{keyinq5}\\
&\|u\|_{L^{\f{2n}{n-2s}}(B(2/3))}\leq C\B(\int_{B^{\ast}(1)}y^{1-2s}|\nabla^{\ast} u^{\ast}|^{2}\Big)^{1/2}+C
\B(\int_{B(1)}|u|^{2}\B)^{1/2},\label{keyinq6}\\
&\|u\|_{L^{2,\,\f{2n}{n-2s}}(Q(2/3))}\leq C\B(\iint_{Q^{\ast}(1)}y^{1-2s}|\nabla^{\ast} u^{\ast}|^{2}\Big)^{1/2}+C
\B(\sup_{-\mu^{2s}\leq t<0}\int_{B(1)}|u|^{2}\B)^{1/2}.\label{keyinq8}
\end{align}
For the proof, we refer the reader to \cite{[TY1]} and Appendix \ref{appendix}.

%This allows us to
%Local realization of the fractional Laplacian
Now  we  present the definition of the suitable weak solution to the MHD equations \eqref{GMHD}.
\begin{definition}\label{defi}
A  triplet  $(u,\,h,\,p)$ is called a suitable weak solution to the generalized  MHD equations \eqref{GMHD} provided the following conditions are satisfied,
\begin{enumerate}[(1)]
\item $u,h\in L^{\infty}(0,\,T;\,L^{2}(\mathbb{R}^{n}))\cap L^{2}(0,\,T;\,\dot{H}^{\alpha}(\mathbb{R}^{n})),p\in
L^{3/2}(0,T;L^{3/2}(\mathbb{R}^{n})).$\label{SWS1}
\item$(u,~h,~p)$~solves (\ref{GMHD}) in $\mathbb{R}^{n}\times (0,T) $ in the sense of distributions.\label{SWS2}
\item$(u,~h,~p)$ satisfies the following inequality
\begin{align}
&\int_{\mathbb{R}^{n}}(|u|^{2}+|h|^{2})\varphi_{1}(x,t)
+2C\int^{t}_{-r^{2\alpha}}\int_{\mathbb{R}_{+}^{n+1}}  \varphi_{2}(x,y,t) y^{1-2\alpha}
(|\nabla^{\ast}u^{\ast}|^{2}+|\nabla^{\ast}h^{\ast}|^{2})\nonumber\\ \leq& C \int^{t}_{-r^{2\alpha} }\int_{\mathbb{R}_{+}^{n+1}}
(|u^{\ast}|^{2}+|h^{\ast}|^{2})\D(y^{1-2\alpha}\nabla^{\ast}( \varphi_{2}))\nonumber\\&
+\int^{t}_{-r^{2\alpha} }\int_{\mathbb{R}^{n}}(|u|^{2}+|h|^{2})
\Big[\partial_{t}\varphi_{1}+ C\lim\limits_{y\rightarrow0^{+}}y^{1-2\alpha}\partial_{y}( \varphi_{2})\Big]
\nonumber\\&+ \int^{t}_{-r^{2\alpha} }
\int_{\mathbb{R}^{n}}u\cdot\nabla\varphi_{1} (|u|^{2}+|h|^{2}+2p)-2 \int^{t}_{-r^{2\alpha} }
\int_{\mathbb{R}^{n}}h\cdot\nabla\varphi_{1} (u\cdot h)+2\int^{t}_{-r^{2\alpha} }\int_{\mathbb{R}^{n}}u\cdot f\varphi_{1}, \label{loc} 
\end{align}
where $\varphi_{1}(x,t)\in C_{0}^{\infty}(\mathbb{R}^{n}\times (0,T) )$ and $\lim\limits_{y\rightarrow0_{+}}\varphi_{2}(x,y,t)=\varphi_{1}(x,t)$.\label{SWS3}
\end{enumerate}
\end{definition}
A point $(x,t)$ is said to be a regular point of the suitable weak solutions to system \eqref{GMHD}  if one has the
boundedness of $(u,\,h)$ in some neighborhood of $(x,t)$; the remaining points are called  singular point denoted by $\mathcal{S}$.

Before we present the decay type lemmas, according to
the natural scaling property of system \eqref{GMHD}, we  introduce the following dimensionless quantities:
\begin{align}
&E_{\ell}(u,r)=\frac{1}{r^{n+2\alpha-(2\alpha-1)\ell}}\iint_{Q(r)}|u|^{\ell}dx dt,\nonumber\\
&P_{\ell}(r)=\frac{1}{r^{n+2\alpha-(4\alpha-2)\ell}}\iint_{Q(r)}|p-\overline{p}_{r}|^{\ell}dx
dt, \nonumber\\
&F_{\ell}(r)=\frac{1}{r^{n+2\alpha-(4\alpha-1)\ell}}\iint_{Q(r)}|f|^{\ell}dxdt, \nonumber
\end{align}
and
\begin{align}
&E_{\ast}(u,r)=\frac{1}{r^{n+2-4\alpha}}\iint_{Q^{\ast}(r)}y^{1-2\alpha}|\nabla^{\ast} u^{\ast}|^2dxdydt,& &E(u,r)=\sup_{-r^{2\alpha}\leq   t<0}\frac{1}{r^{n+2-4\alpha}}\int_{B(r)}|u|^2dx,\nonumber
\\
&E_{\ast}(r)=E_{\ast}(u,r)+E_{\ast}(h,r),&& E(r)=E(u,r)+E(h,r).\nonumber\vspace{-4.2mm}
\end{align}
Note that, by the H\"older inequality, it is enough  to prove Theorem \ref{the1.4} for the borderline case $ n/\ell +2\alpha/q=2\alpha$. Hence, we introduce the     dimensionless quantities below
\begin{align}
&E_{q,\ell}(u,\,r)=
r^{-1}\Big(\int^{0}_{-r^{2\alpha}}\Big(\int_{B(r)}|u|^{\ell}dx\Big)
^{\f{q}{\ell}}dt\Big)^{\f{1}{q}},\nonumber\\
&E_{q,\ell}(r)=E_{q,\ell}(u,r)+E_{q,\ell}(h,r),\nonumber\\
&P_{q,\ell}(r)=
r^{-(n+2-4\alpha)}\Big(\int^{0}_{-r^{2\alpha}}\Big(\int_{B(r)}|p-\overline{p}_{r}|^{\ell}dx\Big)
^{\f{q}{\ell}}dt\Big)^{\f{1}{q}}.\nonumber
\end{align}
\begin{lemma}\label{ineq}
For $0<\mu\leq\f{1}{2}\rho$,~
there is an absolute constant $C$  independent of  $\mu$ and $\rho$,~ such that
 \begin{align}
 E_{3}(u,\mu) &\leq  C \left(\dfrac{\rho}{\mu}\right)^{\f{6+3n-12\alpha}{2}}
E^{\f{6\alpha-n}{4\alpha}}(\rho)E^{\f{n }{4\alpha}}_{\ast}(\rho)
    +C\left(\dfrac{\mu}{\rho}\right)^{6\alpha-3}E^{3/2}(u,\rho).\label{ineq1/2}
 \end{align}
\end{lemma}\begin{proof}
With the help of the triangle inequality, the H\"older inequality and \eqref{keyinq1}, we see that
\begin{align}\nonumber 
\int_{B(\mu)}|u|^{3} &\leq C\int_{B(\mu)}|u-\bar{u}_{{\rho}}|^{3}
+C\int_{B(\mu)}|\bar{u}_{{\rho}}|^{3} \\
&\leq C\B(\int_{B(\f{\rho}{2})}|u-\bar{u}_{{\rho}}|^{2}\B)
^{\f{6\alpha-n}{4\alpha}}
\B(\int_{B(\f{\rho}{2})}|u-\bar{u}_{{\rho}}|
^{\f{2n}{n-2\alpha}}dx\B)^{\f{ (n- 2\alpha)}{4\alpha}}+
 C\f{\mu^{n}}{\rho^{\f{3n}{2}}}\B( \int_{B(\rho)}|u|^{2}\B)^{3/2}\nonumber
\\
&\leq C\B(\int_{B(\rho)}|u-\bar{u}_{\rho}|^{2}dx\B)
^{\f{6\alpha-n}{4\alpha}}
 \B(\int_{B^{\ast}(\rho)}y^{1-2\alpha}|\nabla^{\ast} u^{\ast}|^{2}dxdy\B)^{\f{n}{4\alpha}}+ C\f{\mu^{n}}{\rho^{\f{3n}{2}}}\B( \int_{B(\rho)}|u|^{2}\B)^{3/2}.\label{lem2.31} 
 \end{align}
 
Integrating in time on $(-\mu^{2\alpha},\,0)$ this inequality, we obtain
 \begin{align}
\iint_{Q(\mu)}|u|^{3}
\leq& C\Big(\sup_{-\mu^{2\alpha}\leq t\leq0}\int_{B(\rho)}|u|^{2}\Big)^{\f{6\alpha-n}{4\alpha}}
\int^{0}_{-\mu^{2\alpha}} \Big(\int_{B^{\ast}(\rho)}y^{1-2\alpha}|\nabla^{\ast} u^{\ast}|^{2}\Big)^{\f{n}{4\alpha}}\nonumber\\&+
  C\f{\mu^{n+2\alpha}}{\rho^{\f{3n}{2}}}\B(\sup_{-\rho^{2\alpha}\leq t\leq0}\int_{B(\rho)}|u|^{2}\B)^{3/2}\nonumber \\
\leq& C\mu^{ \f{4\alpha-n}{2}}\B(\sup_{-\rho^{2\alpha}\leq t\leq0}\int_{B(\rho)}|u |^{2}\B)
^{\f{6\alpha-n}{4\alpha}}
 \B(\iint_{Q^{\ast}(\rho)}y^{1-2\alpha}|\nabla^{\ast} u^{\ast}|^{2}\B)^{\f{n}{4\alpha}}\nonumber \\
 &+
 C\f{\mu^{n+2\alpha}}{\rho^{\f{3n}{2}}}\B(\sup_{-\rho^{2\alpha}\leq t\leq0}\int_{B(\rho)}|u|^{2}\B)^{3/2},\nonumber
 \end{align}
which leads to
$$
E_{3}(u,\mu) \leq  C \left(\dfrac{\rho}{\mu}\right)^{\f{6+3n-12\alpha}{2}}
E^{\f{6\alpha-n}{4\alpha}}(\rho)E^{\f{n }{4\alpha}}_{\ast}(\rho)
    +C\left(\dfrac{\mu}{\rho}\right)
    ^{6\alpha-3}E^{3/2}(u,\rho).
$$
 This achieves the proof of this lemma.
\end{proof}
To prove Theorems \ref{the1.2}, \ref{the1.3} and \ref{the1.4}, we need different decay type estimates involving the pressure.
\begin{lemma}\label{presure}
For $0<\mu\leq\f{1}{8}\rho$, there exists an absolute constant $C$  independent of $\mu$ and $\rho$ such that
\begin{align}
P_{\f{n+2\alpha}{n}}(\mu)&\leq C \left(\dfrac{\rho}{\mu}\right)^{\f{(2\alpha+n)(n+2-4\alpha)}{n}}
E^{\f{2\alpha }{n  }}(\rho)E_{\ast}(\rho)
+C
\left(\f{\mu}{\rho}\right)^{\f{8\alpha^{2}
+2\alpha(n-1)-n}{n}}P_{\f{n+2\alpha}{n}}(\rho),\label{presure2}
\\
P_{q',\ell'}(\mu)&\leq C \left(\dfrac{\rho}{\mu}\right)^{n+2-4\alpha}
E^{\f{1}{q}}(\rho)E^{1-\f{1}{q}}_{\ast}(\rho)
+C
\left(\f{\mu}{\rho}\right)
^{4\alpha-\f{n}{\ell}-1}P_{q',\ell'}(\rho),
\label{wzp}\\
P_{3/2}(\mu)&\leq C\left(\dfrac{\rho}{\mu}\right)^{\f{6+3n-12\alpha}{2}}
E^{\f{6\alpha-n}{4\alpha}}(\rho)E^{\f{n }{4\alpha}}_{\ast}(\rho)+C\left(\f{\mu}{\rho}\right)
^{\f{8\alpha-3}{2}}P_{3/2}(\rho),
\label{presure33}\\
P_{3/2}(\mu)\leq &
C\left(\f{\rho}{\mu}\right)
^{n+3-4\alpha}E(\rho)^{\f{1}{q}}E_{\ast}(\rho)
^{1-\f{1}{q}}E_{q,\ell}(\rho)
+C\left(\f{\mu}{\rho}\right)^{\f{8\alpha-3}{2}}P_{3/2}(\rho),
\label{fns}
\end{align}
where the pair $(q',\ell')$ is  the conjugate index of $(q,\,\ell)$ in \eqref{pq1.1}.
\end{lemma}
\begin{proof}
We consider the usual cut-off function $\phi\in C^{\infty}_{0}(B(\rho/2))$ such that $\phi\equiv1$ on $B(\f{3}{8}\rho)$ with $0\leq\phi\leq1$ and
$|\nabla\phi |\leq C\rho^{-1},~|\nabla^{2}\phi |\leq
C\rho^{-2}.$

Due to the incompressibility, we may write
$$
\partial_{i}\partial_{i}(p\phi)=-\phi \partial_{i}\partial_{j}\big[U_{i,j}-H_{i,j}\big]
+2\partial_{i}\phi\partial_{i}p+p\partial_{i}\partial_{i}\phi
,$$
where $U_{i,j}=(u_{j}- \overline{u_{j}} _{\rho/2})(u_{i}-\overline{u_{i}}_{\rho/2})$ and $H_{i,j}=(h_{j}-\overline{h_{j}}_{\rho/2})(h_{i}-\overline{h_{i}}_{\rho/2})$.
This yields that, for any $x\in B(\f{3}{8}\rho)$,
\be\ba\label{pp}
p(x)=&K \ast \{-\phi \partial_{i}\partial_{j}[U_{i,j}-H_{i,j}]
+2\partial_{i}\phi\partial_{i}p+p\partial_{i}\partial_{i}\phi
\}\\
=&-\partial_{i}\partial_{j}K \ast (\phi [U_{i,j}-H_{i,j}])\\
&+2\partial_{i}K \ast(\partial_{j}\phi[U_{i,j}-H_{i,j}])-K \ast
(\partial_{i}\partial_{j}\phi[U_{i,j}-H_{i,j}])\\
& +2\partial_{i}K \ast(\partial_{i}\phi p) -K \ast(\partial_{i}\partial_{i}\phi p)\\
=: &P_{1}(x)+P_{2}(x)+P_{3}(x).
\ea\ee
Since $\phi(x)=1$ when $x\in B(\rho/4$), we know that
\[
\Delta(P_{2}(x)+P_{3}(x))=0.
\]
 By the interior  estimate of harmonic function
and the H\"older inequality, we see that, for every
$ x_{0}\in B(\rho/8)$,
$$\ba
|\nabla (P_{2}+P_{3})(x_{0})|&\leq \f{C}{\rho^{n+1}}\|(P_{2}+P_{3})\|_{L^{1}(B_{x_{0}}(\rho/8))}
\\
&\leq \f{C}{\rho^{n+1}}\|(P_{2}+P_{3})\|_{L^{1}(B(\rho/4))}\\
&\leq \f{C}{\rho^{n+1}}\rho^{n(1-\f{1}{q})} \|(P_{2}+P_{3})\|_{L^{q}(B(\rho/4))}
,
\ea$$
which in turn implies
$$\|\nabla (P_{2}+P_{3})\|^{q}_{L^{\infty}(B(\rho/8))}\leq C \rho^{-(n+q)}\|(P_{2}+P_{3})\|^{q}_{L^{q}(B(\rho/4))}.$$
The latter inequality together with the  mean value theorem leads to, for any  $\mu\leq \f{1}{8}\rho$,
$$\ba
\|(P_{2}+P_{3})-\overline{(P_{2}+P_{3})}_{\mu}\|^{q}_{L^{q}(B(\mu))}\leq&
C\mu^{n} \|(P_{2}+P_{3})-\overline{(P_{2}+P_{3})}_{\mu}\|^{q}_{L^{\infty}(B(\mu))}\\
\leq& C
\mu^{n+q} \|\nabla (P_{2}+P_{3})\|^{q}_{L^{\infty}(B(\rho/8))}\\
\leq& C\Big(\f{\mu}{\rho}\Big)^{n+q}\|(P_{2}+P_{3})\|^{q}_{L^{q}
(B(\rho/4))}.
\ea$$
Integrating the latter relation in time on    $(-\mu^{2\alpha},\,0)$ for $q=\f{n+2\alpha}{n}$, we infer that
$$
\|(P_{2}+P_{3})-\overline{(P_{2}+P_{3})}_{\mu}\|
^{\f{n+2\alpha}{n}}_{L^{\f{n+2\alpha}{n}}(Q(\mu))}\leq C\Big(\f{\mu}{\rho}\Big)^{n+\f{n+2\alpha}{n}}
\|(P_{2}+P_{3})\|^{\f{n+2\alpha}{n}}_{L^{\f{n+2\alpha}{n}}(Q(\rho/4))}.
$$
Notice that $(P_{2}+P_{3})-\overline{(P_{2}+P_{3})}_{\rho/4}$ is also a harmonic function  on $B(\rho/4)$, then the following estimate is valid
$$\ba
&\|(P_{2}+P_{3})-\overline{(P_{2}+P_{3})}_{\mu}\|
^{\f{n+2\alpha}{n}}_{L^{\f{n+2\alpha}{n}}(Q(\mu))}
\\
\leq & C\Big(\f{\mu}{\rho}\Big)^{n+\f{n+2\alpha}{n}}
\|(P_{2}+P_{3})-\overline{(P_{2}+P_{3})}_{\rho/4}\|
^{\f{n+2\alpha}{n}}
_{L^{\f{n+2\alpha}{n}}(Q(\rho/4))}.
\ea$$
In the light  of   the triangle inequality, we have
$$\ba
&\|(P_{2}+P_{3})-\overline{(P_{2}+P_{3})}_{\rho/4}\|_{L^{\f{n+2\alpha}{n}}(Q(\rho/4))}\\
\leq& \|p-\overline{p}_{\rho/4}\|_{L^{\f{n+2\alpha}{n}}(Q(\rho/4))}
+\|P_{1}-\overline{P_{1}}_{\rho/4}\|_{L^{\f{n+2\alpha}{n}}(Q(\rho/4))}
\\
\leq& C\|p-\overline{p}_{\rho}\|_{L^{\f{n+2\alpha}{n}}(Q(\rho/4))}
+C\|P_{1}\|_{L^{\f{n+2\alpha}{n}}(Q(\rho/4))},
\ea$$
which in turns  yields
\be\label{p2rou}\ba
&\|(P_{2}+P_{3})-\overline{(P_{2}+P_{3})}_{\mu}\|
^{\f{n+2\alpha}{n}}_{L^{\f{n+2\alpha}{n}}(Q(\mu))}\\
\leq& C\Big(\f{\mu}{\rho}\Big)^{n+\f{n+2\alpha}{n}}\Big(\|p-\overline{p}
_{\rho}\|^{\f{n+2\alpha}{n}}_{L^{\f{n+2\alpha}{n}}(Q(\rho))}
+\|P_{1}\|^{\f{n+2\alpha}{n}}_{L^{\f{n+2\alpha}{n}}(Q(\rho/4))}\Big).
\ea
\ee
Utilizing the H\"older inequality and \eqref{keyinq1}, we see that
$$\ba
\int_{B(\rho/2)}|u-\overline{u}_{\rho/2}|^{\f{2(n+2\alpha)}{n}}
&\leq
C\B(\int_{B(\rho/2)}|u-\bar{u}_{\rho/2}|^{2}dx\B)
^{\f{2\alpha}{n}}
 \B(\int_{B(\rho/2)}|u-\bar{u}_{\rho/2}|^{\f{2n}{n-2\alpha}} dx\B)^{\f{n-2\alpha}{n}}\\
 &\leq
C\B(\int_{B(\rho/2)}|u|^{2}dx\B)
^{\f{2\alpha}{n}}
 \B(\int_{B(\rho/2)}|u-\bar{u}_{\rho}|^{\f{2n}{n-2\alpha}} dx\B)^{\f{n-2\alpha}{n}}\\
 & \leq
C\B(\int_{B(\rho)}|u|^{2}dx\B)
^{\f{2\alpha}{n}}
\B(\int_{B^{\ast}(\rho)}y^{1-2\alpha}|\nabla^{\ast} u^{\ast}|^{2}dx\B).
\ea$$
According to the classical Calder\'on-Zygmund theorem and the latter inequality, we get that
\be\label{lem2.4.2}\ba
\int_{B(\rho/4)}|P_{1}(x)|^{\f{n+2\alpha}{n}}dx
\leq& C \int_{B(\rho/2)}|u-\overline{u}_{\rho/2}|^{\f{2(n+2\alpha)}{n}}
+|h-\overline{h}_{\rho/2}|^{\f{2(n+2\alpha)}{n}}dx\\
\leq&  C\B(\int_{B(\rho)}|u|^{2}dx\B)
^{\f{2\alpha}{n}}
 \B(\int_{B^{\ast}(\rho)}y^{1-2\alpha}|\nabla^{\ast} u^{\ast}|^{2}dx\B)\\&
 +C\B(\int_{B(\rho)}|h|^{2}dx\B)
^{\f{2\alpha}{n}}
 \B(\int_{B^{\ast}(\rho)}y^{1-2\alpha}|\nabla^{\ast} h^{\ast}|^{2}dx\B),
 \ea\ee
which obviously implies that, for any  $\mu\leq \f{1}{8}\rho$,
\be\label{lem2.4.3}\ba
\int_{B(\mu)}|P_{1}(x)|^{\f{n+2\alpha}{n}}dx \leq &C\B(\int_{B(\rho)}|u|^{2}dx\B)
^{\f{2\alpha }{n}}
 \B(\int_{B^{\ast}(\rho)}y^{1-2\alpha}|\nabla^{\ast} u^{\ast}|^{2}dx\B)\\&
 +C\B(\int_{B(\rho)}|h|^{2}dx\B)
^{\f{2\alpha}{n}}
 \B(\int_{B^{\ast}(\rho)}y^{1-2\alpha}|\nabla^{\ast} h^{\ast}|^{2}dx\B).
 \ea\ee
Let $\tau=(n+2-4\alpha)(1+\f{2\alpha}{n})$, then, it follows from  \eqref{p2rou}-\eqref{lem2.4.3} that
\be\label{lem2.42}
\ba
&\f{1}{\mu^{\tau}}\iint_{Q(\mu)}|p-\overline{p}_{\mu}|^{\f{n+2\alpha}{n}}
\\
 \leq&   C\B(\f{\rho}{\mu}\B)^{\tau}
 \B(\f{1}{\rho^{n+2-4\alpha}}\sup_{-\rho^{2\alpha}\leq t<0}\int_{B(\rho)}|u|^{2}\B)
^{\f{2\alpha }{n}}
 \B(\f{1}{\rho^{n+2-4\alpha}}\iint_{Q^{\ast}(\rho)}y^{1-2\alpha}|\nabla^{\ast} u^{\ast}|^{2}\B) \\&
 + C\B(\f{\rho}{\mu}\B)^{\tau}
  \B(\f{1}{\rho^{n+2-4\alpha}}\sup_{-\rho^{2\alpha}\leq t<0}\int_{B(\rho)}|h|^{2}\B)
^{\f{2\alpha }{n}}
 \B(\f{1}{\rho^{n+2-4\alpha}}\iint_{Q^{\ast}(\rho)}y^{1-2\alpha}|\nabla^{\ast} h^{\ast}|^{2}\B)
 \\
& +C\left(\f{\mu}{\rho}\right)
^{\f{8\alpha^{2}+2\alpha(n-1)-n}{n}}
\f{1}{\rho^{\tau}}\iint_{Q(\rho)}
|p-\overline{p}_{\rho}|^{\f{n+2\alpha}{n}},
\ea
\ee
which turns out that
\begin{align}
P_{\f{n+2\alpha}{n}}(\mu)&\leq C\left(\dfrac{\rho}{\mu}\right)^{\f{(2\alpha+n)(n+2-4\alpha)}{n}}
E^{\f{2\alpha }{n  }}(\rho)E_{\ast}(\rho)
+C
\left(\f{\mu}{\rho}\right)^{\f{8\alpha^{2}
+2\alpha(n-1)-n}{n}}P_{\f{n+2\alpha}{n}}(\rho).\label{presure4}
\end{align}
In view of
\eqref{keyinq1}, we know that
\be\label{wzi1}\ba
\|u-\overline{u}_{\rho/2}\|_{L^{2q',\,2\ell'}(Q(\rho/2))}
&\leq C  \|u-\overline{u}_{\rho/2}\|^{1-\f{1}{q'}}_{L^{\infty,\,2} (Q(\rho/2))}
\| u-\overline{u}_{\rho/2}\|^{\f{1}{q'}}_{L^{2,\,\f{2n}{n-2\alpha}}(Q(\rho/2))}\\
&\leq C  \|u\|^{\f{1}{q}}_{L^{\infty,\,2}(Q(\rho))}
 \|y^{1/2-\alpha}\nabla^{\ast} u^{\ast}\|_{L^2(Q^{\ast}(\rho))}^{1-\f{1}{q}}
, \ea\ee
where  $\alpha/q' + n/(2\ell') = n/2. $ A slight modification about the proof of \eqref{lem2.42} together with
the latter inequality \eqref{wzi1} gives \eqref{wzp}.

Next, we turn our attention to  proving \eqref{fns}.   It follows from the H\"older inequality and \eqref{wzi1}
that
\be\label{wzi2}\ba
\|u-\overline{u}_{\rho/2}\|^{3}_{L^{3}(Q(\rho/2))}
\leq& C\|u-\overline{u}_{\rho/2}\|_{L^{q,\,\ell}(Q(\rho/2))}
\|u-\overline{u}_{\rho/2}\|^{2}_{L^{2q',\,2\ell'}(Q(\rho/2))}\\
\leq& C\|u\|_{L^{q,\,\ell}(Q(\rho))}
\|u\|^{\f{2}{q}}_{L^{\infty,\,2} (Q(\rho))}
 \|y^{1/2-\alpha}\nabla^{\ast} u^{\ast}\|_{L^2(Q^{\ast}(\rho))}^{2-\f{2}{q}},
 \ea\ee
 where $\alpha/q' + n/(2\ell') = n/2. $
 Along the exact same line as in the proof of \eqref{lem2.42}, we obtain  \eqref{fns}.

It remains to show \eqref{presure33}. Indeed, in the light of the H\"older inequality and \eqref{keyinq1}, we get
$$\ba
\int_{B(\rho/2)}|u-\bar{u}_{{\rho/2}}|^{3}
&\leq C\B(\int_{B(\f{\rho}{2})}|u-\bar{u}_{{\rho/2}}|^{2}\B)
^{\f{6\alpha-n}{4\alpha}}
\B(\int_{B(\f{\rho}{2})}|u-\bar{u}_{{\rho/2}}|
^{\f{2n}{n-2\alpha}}\B)^{\f{ (n- 2\alpha)}{4\alpha}}
\\
&\leq C\B(\int_{B(\rho)}|u|^{2}\B)
^{\f{6\alpha-n}{4\alpha}}
 \B(\int_{B^{\ast}(\rho)}y^{1-2\alpha}|  \nabla ^{\ast} u^{\ast}|^{2}\B)^{\f{n}{4\alpha}}.
 \end{aligned}
$$
Integrating this inequality in time, we deduce that
$$ \begin{aligned}
\int_{-\mu^{2\alpha}}^{0}\int_{B(\mu)}|u-\bar{u}_{{\rho/2}}|^{3}
 &\leq C\Big(\sup_{-\mu^{2\alpha}\leq t\leq0}\int_{B(\rho)}|u|^{2}\Big)^{\f{6\alpha-n}{4\alpha}}
\int^{0}_{-\mu^{2\alpha}} \Big(\int_{B^{\ast}(\rho)}y^{1-2\alpha}|\nabla ^{\ast} u^{\ast}|^{2}\Big)^{\f{n}{4\alpha}}  \\
 &\leq C\mu^{ \f{4\alpha-n}{2}}\B(\sup_{-\rho^{2\alpha}\leq t\leq0}\int_{B(\rho)}|u |^{2}\B)
^{\f{6\alpha-n}{4\alpha}}
 \B(\iint_{Q^{\ast}(\rho)}y^{1-2\alpha}|\nabla ^{\ast} u^{\ast}|^{2}\B)^{\f{n}{4\alpha}}.
 \end{aligned}
$$
With this inequality in hand, the rest proof of \eqref{presure33} is parallel to the one of  \eqref{presure4}.
Thus, the proof of this lemma is   completed.
\end{proof}
A slight variant of the proof of
\cite[Lemma 3.2, p.786]{[CKN]} yields
the following lemma .
\begin{lemma}\label{inp}
For any $\mu\leq1/2\rho$,
a constant $C$ independent of $\mu$ and $\rho$ exists such that  the following estimate is valid
\begin{align}
&\f{1}{\mu^{n+3-4\alpha}}\iint_{\tilde{Q}(\mu)} |u||p-\overline{p}_{\tilde{B}(\mu)}|
\nonumber\\\leq &C\Big(\f{1}{\mu^{n+3-4\alpha}}\iint_{\tilde{Q}(\mu)} |u|^{3}\Big)^{1/3}\Big(\f{1}{\mu^{n+3-4\alpha}}
\iint_{\tilde{Q}(2\mu)} |u|^{3}+|h|^{3}\Big)^{2/3}\nonumber
\\&+C\mu^{4\alpha-1} \Big(\f{1}{\mu^{n+3-4\alpha}}\iint_{\tilde{Q}(\mu)} |u|^{3}\Big)^{1/3}\sup_{-\mu^{2\alpha}\leq t-t_{0}<0}
\int_{2\mu<|y-x_{0}|<\rho}\f{|u|^{2}+|h|^{2}}{|y-x_{0}|^{n+1}}dy\nonumber\\&
+ C\Big(\f{\mu}{\rho}\Big)^{\f{8}{3}\alpha-1}\Big(\f{1}{\mu^{n+3-4\alpha}}
\iint_{\tilde{Q}(\mu)} |u|^{3}\Big)^{1/3} \Big(\f{1}{\rho^{n+3-4\alpha}}\iint_{\tilde{Q}(\rho)} |u|^{3}+|h|^{3}\Big)^{2/3}\nonumber\\&
+C\Big(\f{\mu}{\rho}\Big)^{\f{8}{3}\alpha-1}
\Big(\f{1}{\mu^{n+3-4\alpha}}\iint_{\tilde{Q}(\mu)} |u|^{3}\Big)^{1/3}\Big(\f{1}{\rho^{n+3-4\alpha}}\iint_{\tilde{Q}(\rho)} |p|^{3/2}\Big)^{2/3}.\nonumber
\end{align}
\end{lemma}
\begin{proof}
Replace the cut-off function $\phi$  in \eqref{pp} with $\psi$
chosen such that
$\psi(y)=1~~\text{on}~~ B (x_{0},\,3/4\rho) ,\psi(y)=0~~\text{on}~~ B^{c}(x_{0},\,\rho)$ and $\rho^{k}|\nabla^{k}\psi |\leq C$ ($k=1,\,2$).
First, we may write
\begin{align}
P_{1}=&-\partial_{i}\partial_{j}K \ast (\psi[U_{i,j}-H_{i,j}])\nonumber\\
=& \int_{|y-x_{0}|<2\mu}\partial_{i,j}\B[ {K(|x-y|)}\B]\psi[U_{i,j}-H_{i,j}]
\nonumber\\&+\int_{|y-x_{0}|\geq2\mu}\partial_{i,j}\B[ {K(|x-y|)}\B][\psi(U_{i,j}-H_{i,j})]
\nonumber\\=:&P_{11}+P_{12}.
\nonumber
\end{align}
Due to the classical Calder\'on-Zygmund theorem, for any $\mu\leq 1/2\rho$, we see that
$$\|P_{11}\|_{L^{3/2}(\tilde{B}(\mu))}\leq C (\|u\|^{2}_{L^{3}(\tilde{B}(2\mu))}
+\|h\|^{2}_{L^{3}(\tilde{B}(2\mu))}),$$
which together with the H\"older  inequality  implies that
$$\ba
\int_{\tilde{B}(\mu)}|u||P_{11}
-\overline{P_{11}}_{\tilde{B}(\mu)}|&\leq
\int_{\tilde{B}(\mu)}|u||P_{11}|+\int_{\tilde{B}(\mu)}
|u||\overline{P_{11}}_{\tilde{B}(\mu)}|\\
&\leq \|u\|_{L^{3}(\tilde{B}(\mu))}\|P_{11}\|_{L^{3/2}(B(\mu))}+
\|u\|_{L^{3}(\tilde{B}(\mu))}\|P_{11}\|_{L^{3/2}(\tilde{B}(\mu))}\\
&\leq C\|u\|_{L^{3}(\tilde{B}(\mu))}
(\|u\|^{2}_{L^{3}(\tilde{B}(2\mu))}+\|h\|^{2}_{L^{3}(\tilde{B}(2\mu))}).
\ea$$
By a straightforward computation, for any $|x-x_{0}|\leq \mu$, we have
\begin{align}
|\nabla P_{12}(x)|&\leq C\int_{2\mu<|y-x_{0}|<\rho}\f{|u|^{2}+|h|^{2}}{|y-x_{0}|^{n+1}}dy,\nonumber\\
|\nabla P_{2}(x)|&\leq C\rho^{-(n+1)}\int_{\tilde{B}(\rho)} |u|^{2}+|h|^{2}\leq C\rho^{-(2n+3)/3}\B(\int_{\tilde{B}(\rho)} |u|^{3}+|h|^{3}\B)^{2/3},\nonumber\\
|\nabla P_{3}(x)|&\leq C\rho^{-(n+1)}\int_{\tilde{B}(\rho)} |p|\leq C\rho^{-(2n+3)/3}\B(\int_{\tilde{B}(\rho)} |p|^{3/2}\B)^{2/3}.\nonumber
\end{align}
Using the H\"older inequality and the mean value theorem, we observe that
$$
\ba
\int_{\tilde{B}(\mu)}|u||P_{12}-\overline{P_{12}}_{\tilde{B}(\mu)}|&\leq
C\mu^{2n/3}\Big(\int_{\tilde{B}(\mu)}|u|^{3}\Big)^{1/3}
\sup_{x\in \tilde{B}(\mu)}|P_{12}-\overline{P_{12}}_{\tilde{B}(\mu)}|\\
&\leq
C\mu^{(2n+3)/3}\Big(\int_{\tilde{B}(\mu)}|u|^{3}\Big)^{1/3}
\sup_{x\in \tilde{B}(\mu)}|\nabla P_{12}|\\
&\leq
C\mu^{(2n+3)/3}\Big(\int_{\tilde{B}(\mu)}|u|^{3}\Big)^{1/3}
\int_{2\mu<|y-x_{0}|<\rho}\f{|u|^{2}+|h|^{2}}{|y-x_{0}|^{n+1}}dy.
\ea$$
By similar arguments, we can get
$$
\ba
\int_{\tilde{B}(\mu)}|u||P_{2}-\overline{P_{2}}_{\tilde{B}(\mu)}|
&\leq
C\Big(\f{\mu}{\rho}\Big)^{(2n+3)/3}\Big(\int_{\tilde{B}(\mu)}|u|^{3}\Big)^{1/3}
\Big(\int_{\tilde{B}(\rho)} |u|^{3}+|h|^{3}\Big)^{2/3},\\
\int_{\tilde{B}(\mu)}|u||P_{3}-\overline{P_{3}}_{\tilde{B}(\mu)}|
&\leq
C\Big(\f{\mu}{\rho}\Big)^{(2n+3)/3}\Big(\int_{\tilde{B}(\mu)}|u|^{3}\Big)^{1/3}
\Big(\int_{\tilde{B}(\rho)} |p|^{3/2}\Big)^{2/3}
.\ea
$$
Putting together with the above estimates and integrating in time yield  the desired estimate.
\end{proof}

\section{Induction arguments}
\label{sec3}
\setcounter{section}{3}\setcounter{equation}{0}
Based on the induction arguments developed in \cite{[CKN],[TY1],[TY2]}, this section contains the proof of   Theorem \ref{the1.1}. Before proving Theorem \ref{the1.1}, we present a key proposition, which can be seen as the bridge between the previous step and the next step for the given statement in the induction arguments. Moreover, as a by-product of this proposition,  we obtain a corollary, which help us
to circumvent the
straightforward control of the terms involving the magnetic field $h$ and the extensions $u^{\ast}$, $h^{\ast}$ in the local energy inequality \eqref{loc}
  to conclude the proof of Theorems \ref{the1.3} and \ref{the1.4}.
\begin{prop}\label{keyinindu}
There is a constant $C$ such that the following result holds. For any given $(x_{0},\,t_{0})\in\mathbb{R}^{n}\times \mathbb{R}^{-}$ and $k_{0}\in\mathbb{N}$, we have
for any $k>k_{0}$,
\be\label{eq3.1}\ba
&\sup_{-r_{k}^{2\alpha}\leq t-t_0\leq 0}\fbxo(|u|^{2}+|h|^{2})
+r_{k}^{-n}\iint_{\tilde{Q}^{\ast}_{k}}y^{1-2\alpha}
(|\nabla^{\ast}u^{\ast}|^{2}+|\nabla^{\ast}h^{\ast}|^{2}) \\
\leq& C\sup_{-r_{k_{0}}^{2\alpha}\leq t-t_0\leq 0}\fbxozero(|u|^{2}+|h|^{2})+ Cr_{k_{0}}^{-n}\iint_{\tilde{Q}^{\ast}_{k_{0}}}y^{1-2\alpha}
(|\nabla^{\ast}u^{\ast}|^{2}+|\nabla^{\ast}h^{\ast}|^{2})\\
& +C\sum_{l=k_{0}}^{k}r_{l}^{2\alpha-1}\fqxol\Big(|u|^{3}+|h|^{3}+|u||p-\bar{p}_{l}|) +C\sum^{k}_{l=k_{0}}r_{l}^{\gamma'}\B(\fqxol |u|^{3}\B)^{1/3}\|f\|,\ea
\ee
where
$$(\gamma', \|f\|)=\left\{\ba
& \Big(2\alpha-\f{n+2\alpha}{q},\,\|f\|_{L^{q}(Q(1))}\Big) ,\,~~~~1/2<\alpha<3/4,\\
&\Big(\gamma,\,\|f\|_{M_{2\alpha,\gamma}}\Big),~~~~~~
~~~~~~~~~~~~~~~3/4\leq\alpha<1.
\ea\right.
$$
\end{prop}
\begin{proof}
Without loss of generality, we suppose $(x_{0},t_{0})=(0,0)$.
Let $\Gamma( x,r_{k}^{2\alpha}-t)$ be the fundamental solution of the backward fractional heat equation \be\label{fheq}\Gamma_{t}-(-\Delta)^{\alpha}\Gamma=\Gamma_{t}+ C_{\alpha}\lim\limits_{y\rightarrow0^{+}}y^{1-2\alpha}\partial_{y}\Gamma^{\ast }=0, \ee
 where $\Gamma^{\ast }=\Gamma^{\ast}(x,\,y ,\, r_{k}^{2\alpha}-t)$ is the extension of $\Gamma(x ,\,r_{k}^{2\alpha}-t)$ in the sense \eqref{ellipticprblom}. Then, by means of the Poisson   formula \eqref{PF}, we have
\be\label{pf}\Gamma^{\ast}=
\int_{\mathbb{R}^{n}}\f{y^{2\alpha}\Gamma(\xi,r_{k}^{2\alpha}-t)}{(|x-\xi|^{2}+y^{2})^{\f{n+2\alpha}{2}}}
d\xi.\ee
To proceed further, we list some properties of the test function $\Gamma(r_{k}^{2\alpha}-t,x)$ (whose deduction can be found in \cite{[CZ]}):
\begin{align}
&\int_{\mathbb{R}^{n}}\Gamma(x,r_{k}^{2\alpha}-t)dx=1,\label{property1}\\
&\f{C^{-1}(r_{k}^{2\alpha}-t)}{((r_{k}^{2\alpha}-t)^{\f{1}{2\alpha}}+|x|)^{n+2\alpha}} \leq\Gamma(x,r_{k}^{2\alpha}-t)
\leq \f{C(r_{k}^{2\alpha}-t)}{((r_{k}^{2\alpha}-t)^{\f{1}{2\alpha}}+|x|)^{n+2\alpha}},
\label{property2} \\
 &   |\nabla\Gamma(x,r_{k}^{2\alpha}-t)|
 \leq\f{C(r_{k}^{2\alpha}-t)^{1-\f{1}{2\alpha}}}
 {((r_{k}^{2\alpha}-t)^{\f{1}{2\alpha}}+|x|)
 ^{n+2\alpha}},\label{property3}\\
&|\Gamma(x_{1},r_{k}^{2\alpha}-t)-\Gamma(x_{2},r_{k}^{2\alpha}-t)|\leq  \f{C|x_{1}-x_{2}|^{\delta}
(r_{k}^{2\alpha}-t)^{1-\f{\delta}{2\alpha}}}
{\big((r_{k}^{2\alpha}-t)^{\f{1}{2\alpha}}+ ( |x_{1}|\wedge |x_{2}|)\big)^{n+2\alpha} },~~\delta\in(0,2\alpha \wedge 1).\label{property4}
\end{align}
Consider the smooth cut-off functions below
$$\phi_{1}(x,t)=\left\{\ba
&1,\,~~(x,t)\in Q(r_{k_{0}+1}),\\
&0,\,~~(x,t)\in Q^{c}(\f{3}{2}r_{k_{0}+1});
\ea\right.
$$
and
$$\phi_{2}(y)=\left\{\ba
&1,\,~~ 0\leq y\leq r_{k_{0}+1} ,\\
&0,\,~~ y>\f{3}{2} r_{k_{0}+1};
\ea\right.
$$
 satisfying
$$
0\leq \phi_{1},\,\phi_{2}\leq1,~~~~~~r^{2\alpha}_{k_{0}}|\partial_{t}\phi_{1} (x,t)|
+r^{l}_{k_{0}}|\partial^{l}_{x}\phi_{1}(x,t)|\leq C ~~\text{and}~~~~~r^{l}_{k_{0}}|\partial^{l}_{y}\phi_{2}(y)|\leq C.
$$
Notice that $\lim\limits_{y\rightarrow0}\Gamma^{\ast}\phi_{2}(y)=\Gamma$, therefore, setting $\varphi_{1}=\phi_{1}\Gamma$ and $\varphi_{2}=\phi_{1}\phi_{2}\Gamma^{\ast}$ in the  local energy inequality \eqref{loc}, we see that
\begin{align}
&\int_{\mathbb{R}^{n}}(|u|^{2}+|h|^{2})\phi_{1}(x,t)\Gamma
+2C\int^{t}_{-r^{2\alpha}_{k_{0}}}
\int_{\mathbb{R}_{+}^{n+1}}\phi_{1} \phi_{2}\Gamma^{\ast} y^{1-2\alpha}
(|\nabla^{\ast}u^{\ast}|^{2}+|\nabla^{\ast}h^{\ast}|^{2})\nonumber\\
\leq& C \int^{t}_{-r^{2\alpha}_{k_{0}}}\int_{\mathbb{R}_{+}^{n+1}}
(|u^{\ast}|^{2}+|h^{\ast}|^{2})\D(y^{1-2\alpha}\nabla^{\ast}(\phi_{1} \phi_{2}\Gamma^{\ast}))\nonumber\\&
+\int^{t}_{-r^{2\alpha}_{k_{0}}}\int_{\mathbb{R}^{n}}
(|u|^{2}+|h|^{2})
\Big[(\phi_{1}  \Gamma )_{t}+ C\lim\limits_{y\rightarrow0^{+}}y^{1-2\alpha}\partial_{y}(\phi_{1} \phi_{2}\Gamma^{\ast})\Big]\nonumber\\&
+ \int^{t}_{-r^{2\alpha}_{k_{0}}}
\int_{\mathbb{R}^{n}}u\cdot\nabla(\phi_{1} \Gamma)(|u|^{2}+|h|^{2}+2p)-2\int^{t}_{-r^{2\alpha}_{k_{0}}}
\int_{\mathbb{R}^{n}}h\cdot\nabla(\phi_{1} \Gamma) (u\cdot h)\nonumber\\&+2\int^{t}_{-r^{2\alpha}_{k_{0}}}
\int_{\mathbb{R}^{n}}u\cdot f(\phi_{1} \Gamma).
\label{loc21}\end{align}
First, we present the low bound estimates of the terms on the left hand side of this inequality.
Indeed, with the help of \eqref{property2}, we find
\be\label{1}\Gamma( x,r_{k}^{2\alpha}-t)\geq
\f{C^{-1}(r_{k}^{2\alpha}-t)}{((r_{k}^{2\alpha}-t)
^{\f{1}{2\alpha}}+|x|)^{n+2\alpha}}\geq Cr_{k}^{-n},-r_{k}^{2\alpha}\leq t\leq 0, x\in \tilde{B}_{k},\ee
which means that
$$\int_{B_{k}}(|u|^{2}+|h|^{2})\phi_{1} \Gamma\geq C{\fbx}(|u|^{2}+|h|^{2}).$$
For each $y\in[0,r_{k}],~x\in \tilde{B}_{k}$, the triangle inequality allows us to get $|\xi|\leq2r_{k}$ under the
hypothesis  $|x-\xi|<|y|$, then arguing in the same manner as \eqref{1}, we know that  $\Gamma(\xi,r_{k}^{2\alpha}-t)\geq Cr_{k}^{-n} $ if $|\xi|\leq2r_{k}$.
According to the Poisson formula
\eqref{pf}, we see that
$$\Gamma^{\ast}\geq
\int_{|x-\xi|<|y|}\f{y^{2\alpha}\Gamma
(\xi,r_{k}^{2\alpha}-t)}{(|x-\xi|^{2}+y^{2})^{\f{n+2\alpha}{2}}}
d\xi\geq Cr_{k}^{-n},
$$
which in turns implies
$$
\int^{t}_{-r^{2\alpha}_{k_{0}}}\int_{\mathbb{R}_{+}^{n+1}}\phi_{1} \phi_{2}\Gamma^{\ast} y^{1-2\alpha}
\big(|\nabla^{\ast}u^{\ast}|^{2}+|\nabla^{\ast}h^{\ast}|^{2}\big)\geq r_{k}^{-n}\iint_{Q^{\ast}_{k}}y^{1-2\alpha}
\big(|\nabla^{\ast}u^{\ast}|^{2}+|\nabla^{\ast}h^{\ast}|^{2}\big).$$
Secondly, we turn our attentions to the
right hand side of  \eqref{loc21}.
Since  $\Gamma^{\ast}$ is the extension of  $\Gamma$ in the sense \eqref{ellipticprblom}, namely,
 $\D(y^{1-2\alpha}\nabla^{\ast}\Gamma^{\ast})=0$, we write
$$\ba
&\D(y^{1-2\alpha}\nabla^{\ast}(\phi_{1} \phi_{2}\Gamma^{\ast}))\\
=&(1-2\alpha)y^{-2\alpha}\Gamma^{\ast}
\phi_{1}\partial_{y}\phi_{2}+
y^{1-2\alpha}\Gamma^{\ast} \D(\nabla^{\ast} (\phi_{1} \phi_{2}))+2 y^{1-2\alpha}\nabla^{\ast}\Gamma^{\ast}\nabla^{\ast} (\phi_{1} \phi_{2})\\
=&(1-2\alpha)y^{ -2\alpha }\Gamma^{\ast}\phi_{1}\partial_{y}\phi_{2}\\
&+y^{1-2\alpha}\Gamma^{\ast}\phi_{1}
\partial_{yy}\phi_{2}+
y^{1-2\alpha}\Gamma^{\ast}\phi_{2}\Delta\phi_{1}
 \\&+ 2 y^{1-2\alpha}\phi_{1}\partial_{y}\Gamma^{\ast}\partial_{y}\phi_{2}+2 y^{1-2\alpha}\phi_{2} \nabla\Gamma^{\ast} \nabla\phi_{1}\\
 =:&I_{1}+I_{21}+I_{22}+I_{31}+I_{32}
 .\ea$$
Thanks to the support property of $\partial_{y}\phi_{2}$, the estimates of $I_{1},I_{21},I_{31}$ are easy.
Using the Poisson formula   \eqref{pf} and \eqref{property1}, we arrive at
$$\partial_{y}\phi_{2}\Gamma^{\ast}\leq Cy^{-n-1}\int_{\mathbb{R}^{n}}\Gamma(\xi,r_{k}^{2\alpha}-t)d\xi\leq Cy^{-n-1}.$$
By the latter inequality and \eqref{keyinq5}, we know that
$$\ba
\int^{t}_{-r^{2\alpha}_{k_{0}}}\int_{\mathbb{R}_{+}^{n+1}}
|u^{\ast}|^{2}(I_{1}+I_{21})&\leq C r_{k_{0}}^{-(n+2)}\iint_{Q^{\ast}_{k_{0}}}y^{1-2\alpha}
\big|u^{\ast}|^{2}\\
&\leq C\sup_{-r_{k_{0}}^{2\alpha}\leq t\leq0}\fbxozeroo|u|^{2}+ Cr_{k_{0}}^{-n}\iint_{Q^{\ast}_{k_{0}}}y^{1-2\alpha}
|\nabla^{\ast}u^{\ast}|^{2}.
\ea$$
Likewise,
$$\ba
\int^{t}_{-r^{2\alpha}_{k_{0}}}\int_{\mathbb{R}_{+}^{n+1}}
|h^{\ast}|^{2}(I_{1}+I_{21})
&\leq C\sup_{-r_{k_{0}}^{2\alpha}\leq t\leq0}\fbxozeroo|h|^{2}+ Cr_{k_{0}}^{-n}\iint_{Q^{\ast}_{k_{0}}}y^{1-2\alpha}
|\nabla^{\ast}h^{\ast}|^{2}.
\ea$$
It follows from the Poisson formula \eqref{pf} that
\be\label{pf1}\partial_{y}\Gamma^{\ast}(x,y,r_{k}^{2\alpha}-t)=
\int_{\mathbb{R}^{n}}\f{2\alpha|x-\xi|^{2}y^{2\alpha-1}-ny^{2\alpha+1}}
{(|x-\xi|^{2}+y^{2})^{\f{n+2\alpha}{2}+1}}
 \Gamma(\xi,r_{k}^{2\alpha}-t)
d\xi\leq C y^{-n-1},\ee
which implies
$$\ba
\int^{t}_{-r^{2\alpha}_{k_{0}}}\int_{\mathbb{R}_{+}^{n+1}}
(|u^{\ast}|^{2}&+|h^{\ast}|^{2}) I_{31}\\
\leq &C\sup_{-r_{k_{0}}^{2\alpha}\leq t\leq0}\fbxozeroo(|u|^{2}+|h|^{2})+ r_{k_{0}}^{-n}\iint_{Q^{\ast}_{k_{0}}}y^{1-2\alpha}
(|\nabla^{\ast}u^{\ast}|^{2}+|\nabla^{\ast}h^{\ast}|^{2}),
\ea$$
where we have used \eqref{keyinq5} again.

It is clear that
$$
\Delta\phi_{1}\int_{|\xi|\leq2^{-2}r_{k_{0}}}\f{y^{2\alpha}
\Gamma(\xi,r_{k}^{2\alpha}-t)}
{(|x-\xi|^{2}+y^{2})^{\f{n+2\alpha}{2}}}
d\xi\leq|\Delta\phi_{1}|\int_{\mathbb{R}^{n}}
\f{\Gamma(\xi,r_{k}^{2\alpha}-t)1_{|x-\xi|\geq\f{r_{k_{0}}}{4}}}{(|x-\xi|^{2}+y^{2})
^{\f{n}{2}}}
d\xi\leq C r_{k_{0}}^{-n-2}.
$$
Thanks to \eqref{property2}, we have
$$
\int_{|\xi|>2^{-2}r_{k_{0}}}\f{y^{2\alpha}
\Gamma(\xi,r_{k}^{2\alpha}-t)}
{(|x-\xi|^{2}+y^{2})^{\f{n+2\alpha}{2}}}
d\xi\leq Cr_{k_{0}}^{-n} \int_{\mathbb{R}^{n}}\f{y^{2\alpha}
 }{(|x-\xi|^{2}+y^{2})^{\f{n+2\alpha}{2}}}
d\xi\leq Cr_{k_{0}}^{-n}.$$
Consequently, by \eqref{keyinq5} again, we find
\begin{equation*}
  \begin{split}
    &\int^{t}_{-r^{2\alpha}_{k_{0}}} \int_{\mathbb{R}_{+}^{n+1}}
\big(|u^{\ast}|^{2}+|h^{\ast}|^{2}\big) I_{22} \\
\leq &C\sup_{-r_{k_{0}}^{2\alpha}\leq t\leq0}\fbxozeroo(|u|^{2}+|h|^{2})+C r_{k_{0}}^{-n}\iint_{Q^{\ast}_{k_{0}}}y^{1-2\alpha}
(|\nabla^{\ast}u^{\ast}|^{2}+|\nabla^{\ast}h^{\ast}|^{2}).
  \end{split}
\end{equation*}
Moreover, some   straightforward computations   give
$$\ba\nabla \Gamma^{\ast}(x,y,r_{k}^{2\alpha}-t)
&=-(n+2\alpha)
\int_{\mathbb{R}^{n}}\f{y^{2\alpha}(x-\xi)
\Gamma(\xi,r_{k}^{2\alpha}-t)}
{(|x-\xi|^{2}+y^{2})^{\f{n+2\alpha}{2}+1}}
d\xi\\
&=-(n+2\alpha)
\int_{\mathbb{R}^{n}}\f{y^{2\alpha}(x-\xi)
\big[\Gamma(\xi,r_{k}^{2\alpha}-t)
-\Gamma(x,r_{k}^{2\alpha}-t)\big]}
{(|x-\xi|^{2}+y^{2})^{\f{n+2\alpha}{2}+1}}.
\ea$$
 The latter relation  leads to
$$I_{32}\leq C\phi_{2}y\int_{\mathbb{R}^{n}}\f{
|\nabla \phi_{1}|\big|\Gamma(\xi,r_{k}^{2\alpha}-t)
-\Gamma(x,r_{k}^{2\alpha}-t)\big|}
{(|x-\xi|^{2}+y^{2})^{\f{n+1+2\alpha}{2}}}
d\xi.$$
By the H\"older inequality, for any $y\leq r_{k_{0}}$, we easily verify that
$$\ba
|u^{\ast}(x,\,y)|^{2}&=\B|u(x)+\int^{y}_{0}\partial_{z}u^{\ast}dz\B|^{2}\\
&\leq C|u(x)|^{2}+Cy^{2\alpha}\int^{r_{k_{0}}}_{0}
z^{1-2\alpha}|\nabla^{\ast}u^{\ast}|^{2}dz,
\ea$$
which in turns implies,
abusing notation slightly,
\begin{align}
&\int^{t}_{-r^{2\alpha}_{k_{0}}}\int_{\mathbb{R}_{+}^{n+1}}
|u^{\ast}|^{2} I_{32}\nonumber\\
 \leq& C\iint_{Q^{\ast}_{k_{0}}}I_{32}
\Big(|u(x)|^{2}+ y^{2\alpha}\int^{r_{k_{0}}}_{0}z^{1-2\alpha}
|\nabla^{\ast}u^{\ast}|^{2}dz\Big)
\nonumber\\
\leq& C\iint_{Q^{\ast}_{k_{0}}}
 \Big(|u(x)|^{2}+ y^{2\alpha}\int^{r_{k_{0}}}_{0}z^{1-2\alpha}|\nabla^{\ast}
 u^{\ast}|^{2}dz \Big)
\int_{\mathbb{R}^{n}}\f{y
|\nabla \phi_{1}|\big|\Gamma(\xi,r_{k}^{2\alpha}-t)-
\Gamma(x,r_{k}^{2\alpha}-t)\big|}
{(|x-\xi|^{2}+y^{2})^{\f{n+1+2\alpha}{2}}}
d\xi\nonumber\\
\leq& C\iint_{Q_{k_{0}}}
 |u(x)|^{2}\int_{\mathbb{R}^{n}}\f{
|\nabla \phi_{1}|\big|\Gamma(\xi,r_{k}^{2\alpha}-t)
-\Gamma(x,r_{k}^{2\alpha}-t)\big|}
 {|x-\xi|^{n-1+2\alpha}}
d\xi \nonumber\\&+C r^{2\alpha}_{k_{0}} \iint_{Q^{\ast}_{k_{0}}} y^{1-2\alpha}|\nabla^{\ast}u^{\ast}|^{2}
\int_{\mathbb{R}^{n}}\f{
|\nabla \phi_{1}|\big|\Gamma(\xi,r_{k}^{2\alpha}-t)
-\Gamma(x,r_{k}^{2\alpha}-t)\big|}
 {|x-\xi|^{n-1+2\alpha}}
d\xi,\label{2}
\end{align}
where we have used the fact
$$ \int_{0}^{r_{k_{0}}}\f{dy^{2}}
{(|x-\xi|^{2}+y^{2})^{\f{n+1+2\alpha}{2}}}\leq \f{C }{|x-\xi|^{n-1+2\alpha}}.$$
For any $x$ belonging to the support of  $\nabla\phi_{1}$, it is valid that $|x-\xi|\geq 1/4 r_{k_{0}}$ for either
 $|\xi|\leq1/4 r_{k_{0}}$ or $|\xi|\geq2 r_{k_{0}}$.
Then, we deduce that
\begin{align}
&\bigg(\int_{|\xi|\leq1/4 r_{k_{0}}} +\int_{|\xi|\geq2 r_{k_{0}}}\bigg)\f{
 \big|\Gamma(\xi,r_{k}^{2\alpha}-t)-\Gamma(x,r_{k}^{2\alpha}-t)\big|}
 {|x-\xi|^{n-1+2\alpha}}
d\xi \nonumber\\
\leq&C r_{k_{0}}^{-(n-1+2\alpha)}\int_{|\xi|\leq1/4 r_{k_{0}}}\Gamma(\xi,r_{k}^{2\alpha}-t)d\xi+ C r_{k_{0}}^{-(n-1+2\alpha)} \nonumber\\& +Cr_{k_{0}}^{-(n-1+2\alpha)}\int_{|\xi|\geq  2r_{k_{0}}}\Gamma(\xi,r_{k}^{2\alpha}-t)d\xi
+Cr_{k_{0}}^{-n}\int_{|x-\xi|\geq  1/4 r_{k_{0}}}
\f{1}
 {|x-\xi|^{n-1+2\alpha}}d\xi\nonumber\\
 \leq&C r_{k_{0}}^{-(n-1+2\alpha)},\nonumber
\end{align}
where we have used \eqref{property1} and \eqref{property2} with $\f{3r_{k_{0}}}{4} \geq|x|\geq \f{r_{k_{0}}}{2}$.
In light of \eqref{property4}, for $\delta\in (2\alpha-1,1)$, we see that
\begin{align}
&\int_{1/4 r_{k_{0}}\leq|\xi|\leq2 r_{k_{0}}} \f{
 \big|\Gamma(\xi,r_{k}^{2\alpha}-t)-\Gamma(x,r_{k}^{2\alpha}-t)\big|}
 {|x-\xi|^{n-1+2\alpha}}
d\xi \nonumber\\
\leq&C\int_{1/4 r_{k_{0}}\leq|\xi|\leq2 r_{k_{0}}}
 \f{(r_{k}^{2\alpha}-t)^{1-\f{\delta}{2\alpha}}\Big[(r_{k}^{2\alpha}-t)^{\f{1}{2\alpha}}
 +(|x|\wedge|\xi|)\Big]^{-n-2\alpha}}
 {|x-\xi|^{n-1+2\alpha-\delta}}d\xi\nonumber\\
 \leq&Cr_{k_{0}}^{-n-\delta}\int_{|x-\xi|<4r_{k_{0}}}
 \f{1} {|x-\xi|^{n-1+2\alpha-\delta}}d\xi\nonumber\\
 \leq&C r_{k_{0}}^{-(n-1+2\alpha)}.\nonumber
\end{align}
Substituting the above estimates into
\eqref{2}, we arrive at
$$
\int^{t}_{-r^{2\alpha}_{k_{0}}}\int_{\mathbb{R}_{+}^{n+1}}
|u^{\ast}|^{2} I_{32}
\leq C\sup_{-r_{k_{0}}^{2\alpha}\leq t\leq0}\fbxozeroo|u|^{2}+ r_{k_{0}}^{-n}\iint_{Q^{\ast}_{k_{0}}}y^{1-2\alpha}
|\nabla^{\ast}u^{\ast}|^{2}.
$$
 Likewise, we have
 $$\ba
\int^{t}_{-r^{2\alpha}_{k_{0}}}\int_{\mathbb{R}_{+}^{n+1}}
|h^{\ast}|^{2} I_{32}
\leq& C\sup_{-r_{k_{0}}^{2\alpha}\leq t\leq0}\fbxozeroo|h|^{2}+ r_{k_{0}}^{-n}\iint_{Q^{\ast}_{k_{0}}}y^{1-2\alpha}
|\nabla^{\ast}h^{\ast}|^{2}.
 \ea$$
We deduce from \eqref{fheq} that
$$
(\phi_{1}  \Gamma )_{t}+ \lim\limits_{y\rightarrow0^{+}}y^{1-2\alpha}\partial_{y}(\phi_{1} \phi_{2}\Gamma^{\ast})=\partial_{t}\phi_{1}\Gamma.$$
As the support of  $\partial_{t}\phi_{1}$
is included in
$\tilde{Q}(\f{3r_{k_{0}}}{4})/\tilde{Q}(\f{r_{k_{0}}}{2}),$
from \eqref{property2}, we get
$$\int^{t}_{-r^{2\alpha}_{k_{0}}}\int_{\mathbb{R}^{n}}
(|u|^{2}+|h|^{2})
\Big[(\phi_{1}  \Gamma )_{t}+ \lim\limits_{y\rightarrow0^{+}}y^{1-2\alpha}\partial_{y}(\phi_{1} \phi_{2}\Gamma^{\ast})\Big]
\leq C\sup_{-r_{k_{0}}^{2\alpha}\leq t\leq0}\fbxozeroo(|u|^{2}+|h|^{2}).
$$
It follows from \eqref{property2} and \eqref{property3} that
$$\ba&\Gamma\leq  C r_{l+1}^{-n},~\nabla\Gamma \leq C  r_{l+1}^{-(n+1)} ~~~~\text{on}~~~~~Q_{l}/Q_{l+1}, \\
&\Gamma\leq  C r_{k}^{-n},~\nabla\Gamma \leq C  r_{k}^{-(n+1)} ~~~~\text{on}~~~~~Q_{k},\ea   $$
which yields
\begin{align}
&\iint_{Q_{k_{0}}}u\cdot\nabla(\phi_{1} \Gamma) (|u|^{2}+|h|^{2})\nonumber\\
\leq&\sum^{k-1}_{l=k_{0}}\iint_{Q_{l}/Q_{l+1}}
\big(|u|^{3}+|h|^{3}\big)|\nabla(\phi_{1} \Gamma)|  +\iint_{Q_{k}}(|u|^{3}+|h|^{3})|\nabla(\phi_{1} \Gamma)|   \nonumber\\
\leq& C\sum^{k}_{l=k_{0}}r_{l}^{2\alpha-1}\fqxolo(|u|^{3}+|h|^{3})\nonumber.
\end{align}
Exactly as in the above derivation, we derive from the Young inequality that
\begin{align} \int^{t}_{-r^{2\alpha}_{k_{0}}}
\int_{\mathbb{R}^{n}}h\cdot\nabla(\phi_{1} \Gamma) (u\cdot h)
&\leq  C\sum^{k}_{l=k_{0}}r_{l}^{2\alpha-1}\fqxolo(|u|^{3}+|h|^{3}).\nonumber \end{align}
For any $1/2<\alpha<3/4$, using the H\"older inequality twice, we get
\begin{align}
 \int^{t}_{-r^{2\alpha}_{k_{0}}}
\int_{\mathbb{R}^{n}}f\phi_{1} \Gamma  u
&\leq  C\sum^{k}_{l=k_{0}}r_{l}^{-n}\B(\iint_{Q_{l} } |u|^{3}\B)^{1/3} \B(\iint_{Q_{l} } |f|^{3/2}\B)^{2/3}\nonumber \\
 &\leq  C\sum^{k}_{l=k_{0}}r_{l}^{2\alpha-\f{n+2\alpha}{q}}\B(\fqxolo |u|^{3}\B)^{1/3}\B(\iint_{Q(1) } |f|^{q}\B)^{1/q}. \nonumber
\end{align}
Similarly, for any $3/4\leq\alpha<1$, we arrive at
$$\ba
 \int^{t}_{-r^{2\alpha}_{k_{0}}}
\int_{\mathbb{R}^{n}}f\phi_{1} \Gamma  u
 &\leq  C\sum^{k}_{l=k_{0}}r_{l}^{\gamma}
 \B(\fqxolo |u|^{3}\B)^{1/3}\|f\|_{M_{2\alpha,\gamma}}
 .\ea
 $$
Now, let us turn to the term involving pressure in \eqref{loc21}. Before going further, we
set $\chi_{l}=1$ if $|x|\leq7/8 r_{l}$ and $\chi_{l}=0 $ if $ |x|\geq r_{l}$.
Thanks to the support of $(\chi_{l}-\chi_{l+1})$, we derive $|\nabla((\chi_{l}-\chi_{l+1})\phi_{1} \Gamma)|\leq Cr^{-(n+1)}_{l+1}$ from \eqref{property2} and \eqref{property3}.
With the help  of  \eqref{property2} and \eqref{property3} again, we see that $|\nabla( \chi_{k} \phi_{1} \Gamma)|\leq Cr^{-(n+1)}_{k} $.
Thus, thanks to the divergence free condition, we have
$$\ba
\iint_{Q_{k_{0}}}u\cdot\nabla(\phi_{1} \Gamma) p
=&
\sum^{k-1}_{l=k_{0}}\iint_{Q_{l}}
u\cdot\nabla((\chi_{l}-\chi_{l+1})\phi_{1} \Gamma) (p-\bar{p}_{l})+\iint_{Q_{k }}u\cdot\nabla( \chi_{k} \phi_{1} \Gamma) (p-\bar{p}_{k})
\\
\leq& C
\sum^{k }_{l=k_{0}}r_{l}^{2\alpha-1}\fqxolo
|u||p-\bar{p}_{l}|.
\ea$$
Finally, collecting the estimates we have obtained leads  to \eqref{eq3.1}.
\end{proof}
A slight variant of the above proof provides the following
corollary, which allows us to complete the proof of Theorems
\ref{the1.3} and \ref{the1.4}.
\begin{coro}\label{coro}
For any $\mu\leq1/8\rho$, there exists a constant $C$ such that the following result is valid, under the condition $f=0$,
\begin{align} \nonumber
E(\mu)+E_{\ast}(\mu)\leq& C \B(\f{\mu}{\rho}\B)^{4\alpha-2}
(E(\rho/4)+E_{\ast}(\rho/4))\\&+C \Big(\f{\rho}{\mu}\Big)^{n+3-4\alpha}\rho^{-(n+3-4\alpha)}
\iint_{Q(\rho/4)} |u|^{3}+|u||h|^{2}+|u||p-p_{\rho/4}|.
\label{coro1}\end{align}
\end{coro}
Now, we are in a position to prove Theorem \ref{the1.1}.
\begin{proof}[Proof of Theorem \ref{the1.1}]
In what follows, let   $(x_{0},t_{0})\in Q(1/8)$ and $r_{k}=2^{-k}$.
According to the Lebesgue differentiation theorem, it  suffices to show
\be\label{goal}
\fqxoo |u |^{3}+ |h|^{3}+\min\B\{1,\,r_{k}^{\f{4\alpha-3}{3}}\B\}\fqxoo |u||p-\overline{p}_{k}|
\leq \varepsilon_{1}^{2/3}, ~~k\geq3.\ee
First, we show that the latter inequality is valid for $k=3$.  In fact, by means of interpolation inequality, the Young inequality and \eqref{keyinq8}, we infer that
\be\label{eq3.5}\ba
\B(\iint_{Q(2/3)}|u|^{3}\B)^{1/3}&\leq C
\B(\iint_{Q(2/3)}|u|^{\f{2n+4\alpha}{n}}\B)^{\f{n}{2n+4\alpha}}
\\&\leq C\B(\sup_{-(\f{2}{3})^{2\alpha}\leq t<0}\int_{B(1)}|u|^{2}\B)^{\f{n}{2(n+2\alpha)}}
\B(\int^{0}_{-(\f{2}{3})^{2\alpha}}\B(\int_{B(2/3)}|u|^{\f{2n}{n-2\alpha}}\B)
^{\f{n-2\alpha}{n}}\B)^{\f{\alpha}{n+2\alpha}}\\&\leq
C\B(\sup_{-1\leq t<0}\int_{B(2/3)}|u|^{2}\B)^{1/2}+C
\B(\int^{0}_{-(\f{2}{3})^{2\alpha}}\B(\int_{B(2/3)}|u|^{\f{2n}{n-2\alpha}}\B)
^{\f{n-2\alpha}{n}}\B)^{\f{1}{2}}
\\&\leq C\B(\sup_{-1\leq t<0}\int_{B(1)}|u|^{2}\B)^{1/2}+ C\B(\iint_{Q^{\ast}(1)}y^{1-2\alpha}|\nabla^{\ast} u^{\ast}|^{2}\Big)^{1/2}.
\ea\ee
It turns out that
$$
\fqxoth |u |^{3}+|h |^{3}\leq C\varepsilon_{1}^{3/2}.$$
 Applying Lemma \ref{inp} with $\mu=1/8$ and $\rho=1/2$, we observe  that
$$
\ba
&\iint_{\tilde{Q}_{3}}|u||p-\overline{p}_{3}|\\
\leq
&C \B(\iint_{\tilde{Q}_{3}}|u|^{3}\B)^{1/3}
\B(\iint_{\tilde{Q}_{2}}|u|^{3}+|h|^{3}\B)^{2/3}
\\&+C\B(\iint_{\tilde{Q}_{3}}|u|^{3}\B)^{1/3}
\sup_{-(\f{1}{8})^{2\alpha}\leq t-t_{0}<0}\int_{1/4<|y-x_{0}|<1/2}
\f{|u|^{2}+|h|^{2}}{|y-x_{0}|^{n+1}}dy\\&
+C\B(\iint_{\tilde{Q}_{3}}|u|^{3}\B)^{1/3}
\B(\iint_{\tilde{Q}(1/2)}|u|^{3}+|h|^{3}\B)^{2/3}
+C\B(\iint_{\tilde{Q}_{3}}|u|^{3}\B)^{1/3}
\B(\iint_{\tilde{Q}(1/2)}|p|^{3/2}\B)^{2/3}\\
\leq&C\B(\iint_{ Q(2/3)}|u|^{3}\B)^{1/3}
\B(\iint_{Q(2/3)}|u|^{3}+|h|^{3}\B)^{2/3}\\&
+C\B(\iint_{Q(2/3)}|u|^{3}\B)^{1/3}
\sup_{-(2/3)^{2\alpha}\leq t-t_{0}<0}\int_{ |y |<1} \Big(|u|^{2}+|h|^{2}\Big)dy \\&+C
\B(\iint_{Q(2/3)}|u|^{3}\B)^{1/3}
\B(\iint_{Q(2/3)}|u|^{3}+|h|^{3}\B)^{2/3}
+C\B(\iint_{Q(2/3)}|u|^{3}\B)^{1/3}
\B(\iint_{Q(2/3)}|p|^{3/2}\B)^{2/3},
\ea$$
which means
$$
\fqxoth (|u |^{3}+|h |^{3})+\min\B\{1,\,r_{3}^{\f{4\alpha-3}{3}}\B\}\fqxoth |u||p-\overline{p}_{3}|
\leq \varepsilon_{1}^{2/3}.$$
This proves \eqref{goal} in the case $k=3$.
Now, we assume that, for any $3\leq l\leq k$,
$$
\fqxol |u |^{3}+|h |^{3}+\min\B\{1,\,r_{l}^{\f{4\alpha-3}{3}}\B\}\fqxol |u||p-\overline{p}_{l}|
\leq \varepsilon_{1}^{2/3}.$$
For any $3\leq i\leq k$, by Proposition \ref{keyinindu}, \eqref{cond} and the above  induction hypothesis, we find that
\be\label{keyin}\ba
&\sup_{-r_{i}^{2\alpha}\leq t-t_{0}\leq0}\fbxoi(|u|^{2}+|h|^{2})
+r_{i}^{-n}\iint_{\tilde{Q}^{\ast}_{i}}y^{1-2\alpha}
(|\nabla^{\ast}u^{\ast}|^{2}+|\nabla^{\ast}h^{\ast}|^{2})\\
\leq& C \sup_{-r_{3}^{2\alpha}\leq t-t_{0}\leq0}\fbxoz(|u|^{2}+|h|^{2})+ r_{3}^{-n}\iint_{\tilde{Q}^{\ast}_{3}}y^{1-2\alpha}
(|\nabla^{\ast}u^{\ast}|^{2}+|\nabla^{\ast}h^{\ast}|^{2})\\
& +C\sum_{l=3}^{i}r_{l}^{2\alpha-1}
\fqxol\big(|u|^{3}+|h|^{3}+|u||p-\bar{p}_{l}|\big) +C \sum^{i}_{l=3}r_{l}^{\gamma'}\B(\fqxol |u|^{3}\B)^{1/3}\|f\|
\\\leq&
C  \varepsilon_{1} +C\sum_{l=3}^{i}\max\{r_{l}^{2\alpha-1},\,r_{l}^{\f{2}{3}\alpha}\}
\varepsilon_{1}^{2/3}+C\sum^{i}_{l=3}r_{l}^{\gamma'}\varepsilon_{1}^{11/9}\\
\leq&
C  \varepsilon_{1}^{2/3}. \ea
\ee
Making using of the H\"older inequality and \eqref{keyinq2}, we obtain
$$\ba
\int_{\tilde{B}_{k+1} }|u|^{3}dx
&\leq C\B(\int_{\tilde{B}_{k+1}}|u |^{2}dx\B)^{\f{6\alpha-n}{4\alpha}}
\B(\int_{\tilde{B}_{k+1}}|u |^{\f{2n}{n-2\alpha}}dx\B)^{\f{ (n- 2\alpha)}{4\alpha}}  \\
&\leq C\B(\int_{\tilde{B}_{k}}|u |^{2}\B)^{\f{6\alpha-n}{4\alpha}}
\B[\B(\int_{\tilde{B}_{k}^{\ast}}y^{1-2\alpha}|\nabla^{\ast} u^{\ast}|^{2}\Big)^{1/2}+r_{k}^{-\alpha}
\B(\int_{\tilde{B}_{k}}|u|^{2}\B)^{1/2}\B]^{\f{n}{2\alpha}},
\ea$$
which yields
$$\ba
\iint_{\tilde{Q}_{k+1} }|u|^{3}
\leq& C\B(\sup_{-r^{2\alpha}_{k}\leq t-t_{0}<0}\int_{\tilde{B}_{k}}|u |^{2}\B)^{\f{6\alpha-n}{4\alpha}}
 r_{k}^{2\alpha(1-\f{n}{4\alpha})}
 \B(\iint_{\tilde{Q}_{k}^{\ast}}y^{1-2\alpha}|\nabla^{\ast} u^{\ast}|^{2}\Big)^{\f{n}{4\alpha}}\\
 &+r_{k}^{-\alpha+2\alpha}
\B(\sup_{-r^{2\alpha}_{k}\leq t-t_{0}<0}\int_{\tilde{B}_{k}}|u|^{2}\B)^{3/2}.
\ea$$
This inequality, combined with \eqref{keyin}, implies that
\be\label{last2}\ba
\f{1}{r^{n+2\alpha}_{k+1}}\iint_{\tilde{Q}_{k+1}}
|u|^{3}\leq&  C\B( \f{1}{r^{n}_{k}}\sup_{-r^{2\alpha}_{k}\leq t-t_{0}<0}\int_{\tilde{B}_{k}}|u |^{2}\B)^{3/2}\\&
+C\B( \f{1}{r^{n}_{k}}\sup_{-r^{2\alpha}_{k}\leq t-t_{0}<0}\int_{\tilde{B}_{k}}|u |^{2}\B)^{\f{6\alpha-n}{4\alpha}}\B(r_{k}^{-n}\iint_{\tilde{Q}^{\ast}_{k}}y^{1-2\alpha}
|\nabla^{\ast}u^{\ast}|^{2}\B)^{\f{n}{4\alpha}}\\
\leq& C  \varepsilon_{1}.
\ea
\ee
Arguing in the
same manner as above, we see that
\be\label{last22}\ba
\f{1}{r^{n+2\alpha}_{k+1}}\iint_{\tilde{Q}_{k+1}}
|h|^{3}
 \leq  C  \varepsilon_{1}.
\ea
\ee
Employing Lemma \ref{inp} with $\mu=r_{k+1} $ and $\rho=1/2$, we arrive at
\begin{align}
&\f{1}{r_{k+1}^{n+3-4\alpha}}\iint_{\tilde{Q}_{k+1} } |u||p-\overline{p}_{k+1} |\nonumber
\\\leq &C\Big(\f{1}{r_{k+1}^{n+3-4\alpha}}\iint_{\tilde{Q}_{k+1}} |u|^{3}\Big)^{1/3}\Big(\f{1}{r_{k+1}^{n+3-4\alpha}}\iint_{\tilde{Q}_{k}} |u|^{3}+|h|^{3}\Big)^{2/3}\nonumber
\\&+Cr_{k+1}^{4\alpha-1} \B(\f{1}{r_{k+1}^{n+3-4\alpha}}\iint_{\tilde{Q}_{k+1}} |u|^{3}\B)^{1/3}
\sup_{-r^{2\alpha}_{k+1}\leq t -t_{0}<0}\int_{r_{k}<|y-x_{0}|<\f{1}{2}}
\f{|u|^{2}+|h|^{2}}{|y-x_{0}|^{n+1}}dy\nonumber\\&
+ C\Big( r_{k+1} \Big)^{\f{8}{3}\alpha-1}\Big(\f{1}{r_{k+1}^{n+3-4\alpha}}
\iint_{\tilde{Q}_{k+1}} |u|^{3}\Big)^{1/3} \Big(\iint_{\tilde{Q}(1/2)} |u|^{3}+|h|^{3}\Big)^{2/3}\nonumber\\
&+C\Big( r_{k+1} \Big)^{\f{8}{3}\alpha-1}\Big(\f{1}{r_{k+1}^{n+3-4\alpha}}
\iint_{\tilde{Q}_{k+1}} |u|^{3}\Big)^{1/3}\Big(\iint_{\tilde{Q}(1/2)} |p|^{3/2}\Big)^{2/3}.\nonumber
\end{align}
A simple computation together with \eqref{keyin} yields
$$\ba
\int_{r_{k}<|y-x_{0}|<\f12}\f{|u|^{2}+|h|^{2}}{|y-x_{0}|^{n+1}}dy
&=\sum^{k-1}_{l=1}
\int_{r_{l+1}<|y-x_{0}|<r_{l}}\f{|u|^{2}+|h|^{2}}{|y-x_{0}|^{n+1}}\\
&\leq C\sum^{k}_{l=1}r_{l}^{-1}
\varepsilon^{2/3}_{1},
\ea$$
which implies
\begin{align}
&\f{1}{r_{k+1}^{n+2\alpha}}\iint_{\tilde{Q}_{k+1} } |u||p-\overline{p}_{k+1} |\nonumber
\\\leq &C\Big(\f{1}{r_{k+1}^{n+2\alpha}}\iint_{\tilde{Q}_{k+1}} |u|^{3}\Big)^{1/3}\Big(\f{1}{r_{k}^{n+2\alpha}}\iint_{\tilde{Q}_{k}} |u|^{3}+|h|^{3}\Big)^{2/3}
\nonumber\\&+Cr_{k+1}\varepsilon^{2/3}_{1} \B(\f{1}{r_{k+1}^{n+2\alpha}}\iint_{\tilde{Q}_{k+1}} |u|^{3}\B)^{1/3}
  \sum^{k}_{l=1}r_{l}^{-1}
\nonumber\\&+ C r_{k+1}^{\f{3-4\alpha}{3}}\Big(\f{1}{r_{k+1}^{n+2\alpha}}
\iint_{\tilde{Q}_{k+1}} |u|^{3}\Big)^{1/3} \Big( \iint_{ Q(2/3)} |u|^{3}+|h|^{3}\Big)^{2/3}\nonumber\\
&+Cr_{k+1}^{\f{3-4\alpha}{3}}\Big(\f{1}{r_{k+1}^{n+2\alpha}}
\iint_{\tilde{Q}_{k+1}} |u|^{3}\Big)^{1/3}\Big( \iint_{Q(1)} |p|^{3/2}\Big)^{2/3}\nonumber\\
\leq &C(2\varepsilon_{1}
+r_{k+1}^{\f{3-4\alpha}{3}}\varepsilon^{4/3}_{1}
+r_{k+1}^{\f{3-4\alpha}{3}}\varepsilon_{1}^{4/3})\nonumber\\
\leq &\f{1}{2}\varepsilon_{1}^{2/3}
\max\{1,\,r_{k+1}^{\f{3-4\alpha}{3}}\}\nonumber,
\end{align}
where we have used \eqref{last2}, \eqref{last22}, \eqref{eq3.5} and \eqref{cond}.

Collecting the above bounds, we eventually conclude
that
$$\iint_{\tilde{Q}_{k+1}} \!\!\!\!\!\!\!\!\!\! \! \!\!\!\!\!\!\!\!-\hspace{-0.15cm}-\hspace{-0.15cm}~~\,~
(|u|^{3}+|h|^{3})+\min\{1,\,r_{k+1}^{\f{4\alpha-3}{3}}\}
\iint_{\tilde{Q}_{k+1}} \!\!\!\!\!\!\!\!\!\! \! \!\!\!\!\!\!\!\!-\hspace{-0.15cm}-\hspace{-0.15cm}~~\,~|u| |p-\overline{p}_{k+1}|
\leq  \varepsilon_{1}^{2/3}. $$
This completes the proof of this theorem.
\end{proof}

\section{Proofs of Theorems \ref{the1.2}, \ref{the1.3} and \ref{the1.4}}
\label{sec4}
\setcounter{section}{4}\setcounter{equation}{0}
This section is devoted to proving Theorems \ref{the1.2}, \ref{the1.3} and   \ref{the1.4}. More precisely, based on Theorem \ref{the1.1},   we will exploit   appropriate iteration scheme via local
energy inequality \eqref{loc} with respect to different smallness hypothesis condition in these theorems.
Our fundamental tools are
 the decay type lemmas established in Section \ref{sec2} and Corollary \ref{coro} in Section \ref{sec3}.
Eventually, an iteration argument helps us to finish the proof.

\begin{proof}[Proof of Theorem 1.2]
Consider the smooth cut-off functions below
$$\phi_{1}(x,t)=\left\{\ba
&1,\,(x,t)\in Q(\mu),\\
&0,\,(x,t)\in Q^{c}(2\mu );
\ea\right.
$$
and
$$\phi_{2}(y)=\left\{\ba
&1,\, 0\leq y\leq \mu  ,\\
&0,\, y>2 \mu ;
\ea\right.
$$
 satisfying
$$
0\leq \phi_{1},\,\phi_{2}\leq1,~~~~~~\mu^{2\alpha}|\partial_{t}\phi_{1} (x,t)|
+\mu^{l}|\partial^{l}_{x}\phi_{1}(x,t)|\leq C ~~\text{and}~~~~~\mu^{l}|\partial^{l}_{y}\phi_{2}(y)|\leq C.
$$
Choosing   $\varphi_{1}=\phi_{1}$ and $\varphi_{2}=\phi_{1}\phi_{2}$ in  the local energy inequality  \eqref{loc} and utilizing the incompressible condition, we deduce from  \eqref{keyinq5}  that
\begin{align}
&\sup_{-\mu^{2\alpha}\leq t<0}\int_{B(\mu )}|u|^{2}+|h|^{2}
+2C\iint_{Q^{\ast}(\mu )}  y^{1-2\alpha}
\Big(|\nabla^{\ast}u^{\ast}|^{2}
+|\nabla^{\ast}h^{\ast}|^{2}\Big)\nonumber\\
\leq& C
\mu^{-2\alpha}\iint_{Q(2\mu)}(|u|^{2}+|h|^{2})+ C\iint_{Q^{\ast}(2\mu)}
y^{1-2\alpha}(|\nabla^{\ast} u^{\ast}|^{2}+|\nabla^{\ast} h^{\ast}|^{2})  \nonumber\\ &
 + C \mu^{-1}\iint_{Q (2\mu)}  (|u|\big||u|^{2}-|\bar{u}_{\rho}|^{2}\big||+2|u||p-\bar{p}_{\rho}|)\nonumber\\
 & + C \mu^{-1}\iint_{Q(2\mu)}|h|\big|u\cdot h- \overline{u}_{\rho}\cdot \overline{h}_{\rho}  \big|
 +2\iint_{Q(2\mu)}uf.
\label{loc1.2} \end{align}
According to the H\"older inequality,  \eqref{keyinq1} and \eqref{keyinq2},  for any $\mu\leq \rho/4$, we know that
\begin{align}
&\int_{B (2\mu)}|u|\big||u|^{2}-|\bar{u}_{\rho}|^{2}\big| \nonumber\\
\leq & C\mu^{\f{ (4\alpha-n)}{2 }}\B(\int_{B(2\mu)}|u|^{2n/(n-2\alpha)}\B)^{\f{n-2\alpha}{2n}}
\B(\int_{B (2\mu)}|u+\overline{u}_{\rho}|^{2}\B)^{1/2}
\B (\int_{B (2\mu)}|u-\overline{u}_{\rho}|^{2n/(n-2\alpha)}\B)^{\f{n-2\alpha}{2n}}
\nonumber\\
\leq & C\mu^{\f{ (4\alpha-n)}{2 }}\B(\int_{B(\rho/2)}|u|^{2n/(n-2\alpha)}\B)^{\f{n-2\alpha}{2n}}
\B(\int_{B (\rho/2)}|u+\overline{u}_{\rho}|^{2}\B)^{1/2}
\B (\int_{B (\rho/2)}|u-\overline{u}_{\rho}|^{2n/(n-2\alpha)}\B)^{\f{n-2\alpha}{2n}}
\nonumber\\
\leq & C\mu^{\f{ (4\alpha-n)}{2 }}\B\{\B(\int_{B^{\ast}(\rho)}y^{1-2\alpha}|\nabla^{\ast} u^{\ast}|^{2} \Big)^{1/2}\nonumber\\
&+\rho^{-\alpha}
\B(\int_{B(\rho)}|u|^{2}\B)^{1/2}\B\}
\B(\int_{B (\rho)}|u|^{2}\B)^{1/2}
\B (\int_{B^{\ast}(\rho)}
y^{1-2\alpha}|\nabla^{\ast} u^{\ast}|^{2}\B)^{1/2}\nonumber
\\
\leq & C\mu^{\f{ (4\alpha-n)}{2 }} \B(\int_{B^{\ast}(\rho)}y^{1-2\alpha}|\nabla^{\ast} u^{\ast}|^{2} \Big)\B(\int_{B (\rho)}|u|^{2}\B)^{1/2}\nonumber\\
&+C\rho^{-\alpha}\mu^{\f{ (4\alpha-n)}{2 }}
\B(\int_{B(\rho)}|u|^{2}\B)
\B (\int_{B^{\ast}(\rho)}
y^{1-2\alpha}|\nabla^{\ast} u^{\ast}|^{2}\B)^{1/2}.\nonumber
\end{align}
From the triangle inequality, arguing as the latter inequality, we infer that
 \begin{align}
&\int_{B(2\mu )}|h|\big|u\cdot h- \overline{u}_{\rho}\cdot \overline{h}_{\rho} \big| \nonumber\\
\leq & \int_{B(2\mu )}|h u|\big|  h-   \overline{h}_{\rho}\big|+ |h|\big|u -
 \overline{u}_{\rho}   \big||\overline{h}_{\rho}|
\nonumber\\
 \leq &
C\mu^{\f{ (4\alpha-n)}{2 }}\B(\int_{B(\rho/2)}|h|^{2n/(n-2\alpha)}\B)^{\f{n-2\alpha}{2n}}
\B(\int_{B(\rho/2)}|u|^{2}\B)^{1/2}
\B(\int_{B(\rho/2)}|h-\overline{h}_{\rho}|^{2n/(n-2\alpha)}\B)^{\f{n-2\alpha}{2n}}
\nonumber\\
&+
C\mu^{\f{ (4\alpha-n)}{2 }}\B(\int_{B(\rho/2)}|h|^{2n/(n-2\alpha)}\B)^{\f{n-2\alpha}{2n}}
\B(\int_{B(\rho/2)}|\overline{h}_{\rho}|^{2}\B)^{1/2}
\B (\int_{B(\rho/2)}|u-\overline{u}_{\rho}|^{2n/(n-2\alpha)}\B)^{\f{n-2\alpha}{2n}}
\nonumber\\
\leq & C\mu^{\f{ (4\alpha-n)}{2 }} \B(\int_{B^{\ast}(\rho)}y^{1-2\alpha}(|\nabla^{\ast} u^{\ast}|^{2} +|\nabla^{\ast} h^{\ast}|^{2})\Big)\B(\int_{B(\rho)}(|u|^{2}+|h|^{2})\B)^{1/2} \nonumber\\
&+C\rho^{-\alpha}\mu^{\f{ (4\alpha-n)}{2 }}
\B(\int_{B(\rho)}|u|^{2}+|h|^{2}\B)
\B (\int_{B^{\ast}(\rho)}
y^{1-2\alpha}(|\nabla^{\ast} u^{\ast}|^{2}+|\nabla^{\ast} h^{\ast}|^{2})\B)^{1/2}.
 \end{align}
As a consequence, we know that
\begin{align}
&\f{1}{\mu^{n+3-4\alpha}}\iint_{Q(\mu)}|u|\big||u|^{2}-|\bar{u}_{\rho}|^{2}\big|+|h|\big|u\cdot h- \overline{u}_{\rho}\cdot \overline{h}_{\rho} \big| \nonumber\\
\leq & C\B(\f{\rho}{\mu}\B)^{\f{3}{2}(n+2-4\alpha)} \B(\f{1}{\rho^{n+2-4\alpha}}\iint_{Q^{\ast}(\rho)}y^{1-2\alpha}\big(|\nabla^{\ast} u^{\ast}|^{2}+|\nabla^{\ast} h^{\ast}|^{2}\big) \Big)\nonumber\\
&\times\B(\f{1}{\rho^{n+2-4\alpha}}\sup_{-\rho^{2\alpha}\leq t<0}\int_{B (\rho)}(|u|^{2}+|h|^{2})\B)^{1/2} \nonumber\\&+C\B(\f{\rho}{\mu}\B)^{\f{3}{2}(n+2-4\alpha)-\alpha}
\B(\f{1}{\rho^{n+2-4\alpha}}\sup_{-\rho^{2\alpha}\leq t<0}\int_{B(\rho)}|u|^{2}+|h|^{2}\B)\nonumber\\
&\times\B (\f{1}{\rho^{n+2-4\alpha}}\iint_{Q^{\ast}(\rho)}
y^{1-2\alpha}\big(|\nabla^{\ast} u^{\ast}|^{2}+|\nabla^{\ast} h^{\ast}|^{2}\big)\B)^{1/2}\nonumber\\
\leq & C
\B(\f{\rho}{\mu}\B)^{\f{3}{2}(n+2-4\alpha)} \B(E(\rho)^{1/2} E_{\ast}(\rho)+ E(\rho)E^{1/2}_{\ast}(\rho)\B).\label{u3}
\end{align}

For any $1/2<\alpha<3/4$, we find that
\be\label{fff}\ba
F_{q}(\mu)\leq  C\mu^{\f{2\alpha-1}{2\alpha}(2\alpha+n)}\B(\iint_{\mathbb{R}^{n}}|f(x,t)|^{q}dx dt\Big),~~q>\f{2\alpha+n}{2\alpha},
\ea\ee
which means  that $F_{q}(\mu)$  tends to  $0$ as $\mu \rightarrow 0$.
For $\alpha\geq3/4$, we see that
$$
\mu^{2\alpha-1+\gamma}\|f\|_{M_{2\alpha,\gamma}}
\rightarrow0 ~~\text{as}~~\mu\rightarrow0.
$$
Thus the smallness of external force has been shown.
Next, we will prove the smallness of the other terms in \eqref{cond}.
First, by the H\"older inequality, for any $\alpha\geq3/4$, it holds
$$\ba
F_{3/2}(\mu)&=\mu^{3/2(4\alpha-1) -(2\alpha+n)}\iint_{Q(\mu)}|f(x,t)|^{3/2}dx dt\\
\leq& C\mu^{\f{3\gamma}{2}+3\alpha-\f{3}{2}}
\Big(\f{1}{\mu^{2\alpha(\gamma-2\alpha)}}\fqxomu|f(x,t)|^{2\alpha}dx dt\Big)^{3/4\alpha}\\
\leq& C\mu^{\f{3\gamma}{2}+3\alpha-\f{3}{2}}
\|f\|^{3/2}_{M_{2\alpha,\gamma}}.
\ea
$$
Therefore, we see that $F_{3/2}(\mu)\rightarrow0$ as $\mu\rightarrow0$ when $\alpha\geq3/4$.
By similar arguments, the same statement is also valid if  $1/2<\alpha<3/4$.
Therefore, we see that there is a constant $\mu_{1}$ such that for any $\mu\leq \mu_{1}$, $F_{3/2}(\mu)\leq\varepsilon_{2}$.
The condition   \eqref{parabolic} yields
that there exists $\mu_{0}\in(0,\,\mu_1)$ such that
\be\label{condition1}E_{\ast}(\mu)\leq \varepsilon_{2},~\text{for any} ~~\mu\leq \mu_{0}.\ee
Plugging \eqref{u3},  \eqref{presure2} in Lemma \ref{presure} into \eqref{loc1.2} and using
the Cauchy-Schwarz inequality, we have
\begin{align}
&E(\mu )+E_{\ast}(\mu )
\nonumber\\\leq & C
E_{3}^{2/3}(2\mu)+CE_{\ast}(2\mu)+
C
\B(\f{\rho}{\mu}\B)^{\f{3}{2}(n+2-4\alpha)} \B(E(\rho)^{1/2} E_{\ast}(\rho)+ E(\rho)E^{1/2}_{\ast}(\rho)\B)\nonumber\\
&+ CE^{1/3}_{3}(u,2\mu)P^{2/3}_{3/2}(2\mu)+CE^{1/3}_{3}(u,2\mu)F^{2/3}_{3/2}(2\mu)
\nonumber\\\leq& C
E_{3}^{2/3}( 2\mu)+CE_{\ast}(2\mu)+
C\B(\f{\rho}{\mu}\B)^{\f{3}{2}(n+2-4\alpha)} \B(E(\rho)^{1/2} E_{\ast}(\rho)+ E(\rho)E^{1/2}_{\ast}(\rho)\B)\nonumber\\
 &+ P^{\f{2n}{n+2\alpha}}_{\f{n+2\alpha}{n}}( 2\mu)+F^{4/3}_{3/2}(2\mu)
\nonumber\\
\leq&  C \left(\dfrac{\rho}{\mu}\right)^{2+n-4\alpha}
E^{\f{6\alpha-n}{6\alpha}}(\rho)E^{\f{n }{6\alpha}}_{\ast}(\rho)
    +C\left(\dfrac{\mu}{\rho}\right)^{4\alpha-2}E(\rho)\nonumber\\&
    +C\B(\f{\rho}{\mu}\B)^{n+2-4\alpha} E_{\ast}(\rho)+
\B(\f{\rho}{\mu}\B)^{\f{3}{2}(n+2-4\alpha)} \B(E(\rho)^{1/2} E_{\ast}(\rho)+ E(\rho)E^{1/2}_{\ast}(\rho)\B)\nonumber\\& + C\left(\dfrac{\rho}{\mu}\right)^{2n+4-8\alpha}
E^{\f{4\alpha }{n+2\alpha  }}(\rho)E^{\f{2n}{n+2\alpha}}_{\ast}(\rho)\nonumber\\
&+C\left(\f{\mu}{\rho}\right)
^{\f{2(8\alpha^{2}+2\alpha(n-1)-n)}{n+2\alpha}}
P^{\f{2n}{n+2\alpha}}_{\f{n+2\alpha}{n}}(\rho)
+\B(\f{\rho}{\mu}\B)^{\f{4}{3}n+2-\f{16\alpha}{3}}F^{4/3}_{3/2}(\rho),
\label{4.6}\end{align}
where we have used \eqref{ineq1/2} in Lemma \ref{ineq}.
Before going further, we set
$$G(\mu)=E(\mu)+E_{\ast}(\mu)
+P^{\f{2n}{n+2\alpha}}_{\f{n+2\alpha}{n}}(\mu).$$
It follows from \eqref{4.6} and \eqref{presure2} in Lemma \ref{presure} that
\begin{align}
 G(\mu)
\leq&   C \left(\dfrac{\rho}{\mu}\right)^{2+n-4\alpha}
E^{\f{6\alpha-n}{6\alpha}}(\rho)E^{\f{n }{6\alpha}}_{\ast}(\rho)
    +C\left(\dfrac{\mu}{\rho}\right)^{4\alpha-2}E(\rho)
    \nonumber\\&
+\B(\f{\rho}{\mu}\B)^{\f{3(n+2-4\alpha)}{2}}
\B[G^{1/2}(\rho)E_{\ast}(\rho)
+G(\rho)E^{1/2}_{\ast}(\rho)\B]
\nonumber\\ &  +\B(\f{\rho}{\mu}\B)^{n+2-4\alpha}E_{\ast}(\rho)
+ \B(\f{\rho}{\mu}\B)^{2n}F^{4/3}_{3/2}(\rho)     \nonumber\\ &    +\left(\dfrac{\rho}{\mu}\right)^{2n+4-8\alpha}
E^{\f{4\alpha }{n+2\alpha  }}(\rho)E^{\f{2n}{n+2\alpha}}_{\ast}(\rho)
+C\left(\f{\mu}{\rho}\right)
^{\f{2(8\alpha^{2}+2\alpha(n-1)-n)}{n+2\alpha}}
P^{\f{2n}{n+2\alpha}}_{\f{n+2\alpha}{n}}(\rho).\nonumber
\end{align}
Notice that $\f{4\alpha }{n+2\alpha}<1$ and $8\alpha^{2}+2\alpha(n-1)-n>0$ for $n/4\leq\alpha<n/2.$ Then, the latter inequality yields
$$
 G(\mu)
\leq C \B(\f{\rho}{\mu}\B)^{2(n+2-4\alpha)}
 G^{\tau_{1}}(\rho)E^{\tau_{2}}_{\ast}(\rho)
 +\B(\f{\mu}{\rho}\B)^{^{\tau_{3}}} G(\rho)+ \B(\f{\rho}{\mu}\B)^{2n}F^{4/3}_{3/2}(\rho),
$$
where $0\leq\tau_{1}<1,\,\tau_{2},\,\tau_{3}>0$.
Using the Young inequality, we infer that
$$\ba
 G(\mu)
&\leq C \B(\f{\rho}{\mu}\B)^{6}
 G(\rho)E^{\tau_{4}}_{\ast}(\rho)+
 \B(\f{\mu}{\rho}\B)^{^{\tau_{3}}} G(\rho)+\B(\f{\rho}{\mu}\B)^{6}(E^{\tau_{4}}_{\ast}(\rho)
 +F^{4/3}_{3/2}(\rho))\\
&\leq C_{1}\lambda^{-6}\varepsilon_{2}^{\tau_{4}}G(\rho)+
C_{2}\lambda^{\tau_{3}}G(\rho)+C_{3}\lambda^{-6}
\varepsilon_{2}^{\tau_{5}},
\ea
$$
where  ${\alpha}>0 , \,\tau_{4},\,\tau_{5}>0,\,\lambda=\f{\mu}{\rho}\leq \f{1}{32}$ and $\rho\leq \mu_{0}$.

Choosing $\lambda,~\varepsilon_{2}$ such that $\bar{q}=2C_{2}\lambda^{\tau_{3}}<1$ and
$$
\varepsilon_{2}=\min\bigg\{\B(\f{\bar{q}\lambda^{6}}{2C_{1}}\B)^{1/\tau_{4}},
\B(\f{(1-\bar{q})\lambda^{10+2n-8\alpha}\varepsilon_{1}}{2C_{3}}\B)
^{\f{1}{ \tau_{5}}}\bigg\},
$$
we obtain
\be\label{iterr}
G(\lambda\rho)\leq \bar{q}G(\rho)+C_{3}\lambda^{-6}\varepsilon_{2}^{\tau_{5}}. \ee
Iterating $\eqref{iterr}$, we deduce that
 \[
G(\lambda^{k}\rho)\leq \bar{q}^{k}G(\rho)+\f{1}{2}\lambda^{2n+4-8\alpha}\varepsilon_{1}. \]
From the definition of $G(\mu)$, there exists a positive number $ K_{0}$~such that $$\bar{q}^{K_{0}}G(\mu_{0})\leq 4 \f{M(\|u\|_{L^{\infty}L^{2}},\|u\|_{L^{2}H^{\alpha}},
\|p\|_{L^{3/2}L^{3/2}})}{\mu_{0}^{n+2-4\alpha}}\bar{q}^{K_{0}}
\leq\dfrac{1}{2}\varepsilon_{1}\lambda^{2n+4-8\alpha}.$$
Let  $\mu_{2}:=\lambda^{K_{0}}\mu_{0}$, ~therefore for all $0<\mu\leq \mu_{2}$, $\exists k\geq
K_{0}$,~such that $\lambda^{k+1}\mu_{0}\leq \mu\leq \lambda^{k} \mu_{0}$.
Thus, we have
 \begin{align}
&E(\mu)+E_{\ast}(\mu)+P^{\f{2n}{n+2\alpha}}_{\f{n+2\alpha}{n}}(\mu)\nonumber\\
\leq& \sup_{-(\lambda^{k}\mu_{0})^{2\alpha}\leq   t<0}\frac{1}{(\lambda^{k+1}\mu_{0})^{n+2-4\alpha}}\int_{B(\lambda^{k}\mu_{0})}(|u|^2+|h|^{2})\nonumber\\
&+\frac{1}{(\lambda^{k+1}\mu_{0})^{n+2-4\alpha}}\iint_{Q^{\ast}(\lambda^{k}\mu_{0})}y^{1-2\alpha}(|\nabla^{\ast} u^{\ast}|^2+|\nabla^{\ast} h^{\ast}|^2)\nonumber\\
&+
\frac{1}{(\lambda^{k+1}\mu_{0})^{2n+4-8\alpha}}
\B(\iint_{Q(\lambda^{k}\mu_{0})}|p|^{\frac{n+2\alpha}{n}}dxdt\B)^{2n/(n+2\alpha)}
\nonumber\\
\leq & \f{1}{\lambda^{2n+4-8\alpha}}G(\lambda^{k}\mu_{0}) \nonumber\\
\leq &\f{1}{\lambda^{2n+4-8\alpha}}(\bar{q}^{k-K_{0}}\bar{q}^{K_{0}}
G(\mu_{0})+\f{1}{2}\lambda^{2n+4-8\alpha}\varepsilon_{1})\nonumber\\
\leq &\varepsilon_{1}\label{itera}.
 \end{align}
By a  scaling argument together with Theorem \ref{the1.1}, we end the proof of Theorem \ref{the1.2}.
\end{proof}

\begin{proof}[Proof of Theorem  \ref{the1.3}]
Thanks to \eqref{ineq1/2} in Lemma \ref{ineq} and \eqref{presure33} in Lemma \ref{presure}, we see that
$$\ba &E^{2/3}_{3}(\rho/4) \leq  C
(E (\rho)+E_{\ast}(\rho))
    +CE(\rho),\\
&P_{3/2}^{2/3}(\rho/4)\leq  C
\B(E (\rho)+E_{\ast}(\rho)\B)+C P_{3/2}^{2/3}(\rho).
\ea
$$
By the H\"older inequality, \eqref{coro1} in Corollary \ref{coro} and the latter inequalities, we see that
\begin{align}
E(\mu)+E_{\ast}(\mu)
\leq& C    \B(\f{\mu}{\rho}\B)^{4\alpha-2}\B(E (\rho/4)+E_{\ast}(\rho/4)\B)+
 \Big(\f{\rho}{\mu}\Big)^{n+3-4\alpha}E_{3}(u,\rho/4)\nonumber\\&
+ \Big(\f{\rho}{\mu}\Big)^{n+3-4\alpha}E^{1/3}_{3}(u,\rho/4)
E^{2/3}_{3}(h,\rho/4)+\Big(\f{\rho}{\mu}\Big)^{n+3-4\alpha}E^{1/3}_{3}(u,\rho/4)
P^{2/3}_{ 3/2}(\rho/4)\nonumber\\
\leq&    C \B(\f{\mu}{\rho}\B)^{4\alpha-2}\B(E ( \rho)+E_{\ast}(  \rho)\B)+
 \Big(\f{\rho}{\mu}\Big)^{n+3-4\alpha}E_{3}(u, \rho)\nonumber\\&
+ \Big(\f{\rho}{\mu}\Big)^{n+3-4\alpha}E^{1/3}_{3}(u,\rho)
\B(E ( \rho)+E_{\ast}( \rho)\B) \nonumber\\
&+\Big(\f{\rho}{\mu}\Big)^{n+3-4\alpha}E^{1/3}_{3}(u,\rho)
\B(E ( \rho)+E_{\ast}( \rho)+ P_{3/2}^{2/3}( \rho)\B)\label{4.10}.
 \end{align}
With the help of \eqref{presure33} again and Young's inequality, we infer that
$$\ba
P_{3/2}^{2/3}(\mu) &\leq \left(\dfrac{\rho}{\mu}\right)^{\f{6+3n-12\alpha}{3}}
E^{\f{6\alpha-n}{6\alpha}}(\rho)E^{\f{n }{6\alpha}}_{\ast}(\rho)+C\left(\f{\mu}{\rho}\right)
^{\f{8\alpha-3}{3}}P_{3/2}^{2/3}(\rho)  \\
&\leq \left(\dfrac{\rho}{\mu}\right)^{ 2+ n-4\alpha }
\B(E(\rho)+E_{\ast}(\rho)\B)+C\left(\f{\mu}{\rho}\right)
^{\f{8\alpha-3}{3}}
P^{2/3}_{3/2}(\rho).
\ea$$
Now, set
$$
G_{1}(\mu)=E (\mu)+E_{\ast}(\mu)+ \varepsilon_{3}^{1/4}P_{3/2}^{2/3}(\mu).
$$
It follows from \eqref{4.10} and the latter inequality that
\begin{align}
G_{1}(\mu)\leq & C \B(\f{\mu}{\rho}\B)^{4\alpha-2}G_{1}(\rho)+
 \Big(\f{\rho}{\mu}\Big)^{n+3-4\alpha}E_{3}(u,\rho)
+ \Big(\f{\rho}{\mu}\Big)^{n+3-4\alpha}E^{1/3}_{3}(u,\rho)
G_{1}(\rho)\nonumber\\
&+\Big(\f{\rho}{\mu}\Big)^{n+3-4\alpha}E^{1/3}_{3}(u,\rho)
\B(G_{1}(\rho)+\varepsilon_{3}^{-1/4}G_{1}(\rho)\B)\nonumber\\&
+\left(\dfrac{\rho}{\mu}\right)^{2+ n-4\alpha}
\varepsilon_{3}^{1/4}G_{1}(\rho)+C\left(\f{\mu}{\rho}\right)
^{\f{8\alpha-3}{3}}G_{1}(\rho)\nonumber\\
\leq & C \Big(\f{\rho}{\mu}\Big)^{n+3-4\alpha}\varepsilon_{3}^{1/12}
 G_{1}(\rho)+ C \B(\f{\mu}{\rho}\B)^{\tau_{6}}G_{1}(\rho)+
 \Big(\f{\rho}{\mu}\Big)^{n+3-4\alpha} \varepsilon_{3}
 \nonumber\\
\leq & C \lambda^{4\alpha -n-3}\varepsilon_{3}^{1/12}
 G_{1}(\rho)+ C \lambda^{\tau_{6}}G_{1}(\rho)+
 \lambda^{4\alpha -n-3} \varepsilon_{3},\nonumber
\end{align}
where $\tau_{6}>0$, $\lambda=\frac{\mu}{\rho}\leq \frac{1}{32}$.
An iteration argument completely analogous to that adopted in the proof of Theorem \ref{the1.2} yields the smallness of $E (\mu)+E_{\ast}(\mu)$.  This together with Theorem
 \ref{the1.2} completes the
 proof of Theorem \ref{the1.3}.
\end{proof}
\begin{proof}[Proof of Theorem \ref{the1.4}]
(1)~Recall that the pair $(q',\ell')$ is  the conjugate index of $(q,\,\ell)$ in \eqref{pq1.1}.
By means of the H\"{o}lder  inequality,  \eqref{keyinq7}  and the Young inequality,  we arrive at
$$\ba
\|u\|_{L^{2q'}_{t}L^{2\ell'}_{x}(Q(\rho/2))}
&\leq C  \|u\|^{1-\f{1}{q'}}_{L_{t}^{\infty}L^{2}_{x} (Q(\rho/2))}
\| u\|^{\f{1}{q'}}_{L_{t}^2L_{x}^{\f{2n}{n-2\alpha}}(Q(\rho/2))}\\
&\leq C  \|u\|^{\f{1}{q}}_{L_{t}^{\infty}L^{2}_{x} (Q(\rho))}
\B(\|y^{1/2-\alpha}\nabla^{\ast} u^{\ast}\|_{L^2(Q^{\ast}(\rho))}+\|u\|_{L^{\infty,2}(Q(\rho))}\B)^{1-\f{1}{q}}
,\ea
$$
where  $\alpha/q' + n/2\ell' = n/2. $
In view of the H\"older inequality and the last inequality, we deduce that
\be\label{wzii1}\ba
\rho^{-(n+3-4\alpha)} \int_{-(\rho/2)^{2\alpha}}^{0}\int_{B(\rho/2)}|u||h|^{2}
\leq& C\rho^{-(n+3-4\alpha)} \|u\|_{L^{q,\ell}(Q(\rho/2))}
\|h\|^{2}_{L^{2q',2\ell'}(Q(\rho/2))}\\
\leq&  C [E(\rho)+E_{\ast}(\rho)] E_{q,\ell}(u,\rho).\ea
\ee
A similar procedure yields
 \be\label{wzii}\ba
\rho^{-(n+3-4\alpha)} \int_{-(\rho/2)^{2\alpha}}^{0}\int_{B(\rho/2)}|u|^{3}
\leq   C [E(\rho)+E_{\ast}(\rho)] E_{q,\ell}(u,\rho)
,\ea
\ee
which together with \eqref{wzii1} and \eqref{coro1} in Corollary \ref{coro}
implies that
\be\label{4.13}\ba
&E(\mu)+E_{\ast}(\mu)\\
\leq&C \B(\f{\mu}{\rho}\B)^{4\alpha-2}
\B(E(\rho/4)+E_{\ast}(\rho/4)\B)
+C\Big(\f{\rho}{\mu}\Big)^{n+3-4\alpha}
\B[E(\rho)+E_{\ast}(\rho)\B] E_{q,\ell}(u,\rho)
\\&+C\Big(\f{\rho}{\mu}\Big)^{n+3-4\alpha}
E^{1/3}_{3}(u,\rho)P^{2/3}_{3/2}(\rho)\\
\leq&C \B(\f{\mu}{\rho}\B)^{4\alpha-2}
\B(E(\rho)+E_{\ast}(\rho)\B)+\Big(\f{\rho}{\mu}\Big)^{n+3-4\alpha}
\B[E(\rho)+E_{\ast}(\rho)\B] E_{q,\ell}(u,\rho)
\\&+\Big(\f{\rho}{\mu}\Big)^{n+3-4\alpha}
E^{1/3}_{q,\ell}(u,\rho)\B[E(\rho)+E_{\ast}(\rho)\B]^{1/3} P^{2/3}_{3/2}(\rho).
\ea
\ee
It follows form \eqref{fns} and condition \eqref{hbd}  that
$$\ba
\varepsilon_{4}^{1/4}P_{3/2}(\mu)\leq &
C\varepsilon_{4}^{1/4}\left(\f{\rho}{\mu}\right)
^{n+3-4\alpha}E(\rho)^{\f{1}{q}}E_{\ast}(\rho)
^{1-\f{1}{q}}E_{q,\ell}(\rho)
+C\left(\f{\mu}{\rho}\right)^{\f{8\alpha-3}{2}}\varepsilon_{4}^{1/4}P_{3/2}(\rho)\\
\leq &
C(M)\varepsilon_{4}^{1/4}\left(\f{\rho}{\mu}\right)
^{n+3-4\alpha}E(\rho)^{\f{1}{q}}E_{\ast}(\rho)
^{1-\f{1}{q}}
+C\left(\f{\mu}{\rho}\right)^{\f{8\alpha-3}{2}}
\varepsilon_{4}^{1/4}P_{3/2}(\rho).
 \ea$$
Set $G_{2}(\mu)=E(\mu)+E_{\ast}(\mu)+\varepsilon_{4}^{1/4}P_{3/2}(\mu)$ and $\lambda=\f{\mu}{\rho} $, then the latter inequality  together with \eqref{4.13} implies that
\begin{align}
 G_{2}(\lambda\rho)\leq& C
\lambda^{4\alpha-2}
\B(E(\rho)+E_{\ast}(\rho)\B)+\lambda^{4\alpha -n-3}
\B[E(\rho)+E_{\ast}(\rho)\B] E_{q,\ell}(u,\rho)
\nonumber\\
&+\lambda^{4\alpha -n-3}
E^{1/3}_{q,\ell}(u,\rho)\varepsilon_{4}^{-\f{1}{6}}G_{2}(\rho)
\nonumber\\
&+C\varepsilon_{4}^{1/4}\left(\f{\rho}{\mu}\right)
^{n+3-4\alpha}E(\rho)^{\f{1}{q}}E_{\ast}(\rho)
^{1-\f{1}{q}}
+C\left(\f{\mu}{\rho}\right)^{\f{8\alpha-3}{2}}
\varepsilon_{4}^{1/4}P_{3/2}(\rho)\nonumber\\
\leq& C\lambda^{4\alpha-2}
 G_{2}(\rho)+\lambda^{4\alpha -n-3}
 G_{2}(\rho) E_{q,\ell}(u,\rho)\nonumber\\
&+\lambda^{4\alpha -n-3}
E^{1/3}_{q,\ell}(u,\rho)\varepsilon_{4}^{-\f{1}{6}}G_{2}(\rho)
\nonumber\\
&+C\varepsilon_{4}^{1/4}\lambda^{4\alpha -n-3}G_{2}(\rho)
+C\lambda^{\f{8\alpha-3}{2}}
G_{2}(\rho)\nonumber\\
\leq&\lambda^{-\tau_{8}}
\varepsilon_{4}^{\f{1}{6}}G_{2}(\rho) +C
\lambda^{\tau_{7}}
 G_{2}(\rho),
 \end{align}
where $\tau_{7},~\tau_{8}>0$.
With this in hand,
 exactly as in the derivation of  \eqref{itera} and using Theorem \ref{the1.2}, one can finish the proof of first part of Theorem \ref{the1.4}.

\noindent
(2) By a slight modified proof of \eqref{4.13} and \eqref{wzp}, we arrive at
 \begin{align}
&E(\mu)+E_{\ast}(\mu)\nonumber\\
\leq&
C \B(\f{\mu}{\rho}\B)^{4\alpha-2}
\B(E(\rho)+E_{\ast}(\rho)\B)+\Big(\f{\rho}{\mu}\Big)^{n+3-4\alpha}
\B[E(\rho)\\&+E_{\ast}(\rho)\B] E_{q,\ell}(u,\rho)+\Big(\f{\rho}{\mu}\Big)^{n+3-4\alpha}E_{q,\ell}(u,\rho)P_{q',\ell'}(\rho)\nonumber\\
\leq&
C \B(\f{\mu}{\rho}\B)^{4\alpha-2}
\B(\f{\varrho}{\rho}\B)^{n+2-4\alpha}
\B(E(\varrho)+E_{\ast}(\varrho)\B)\nonumber\\&+
\Big(\f{\rho}{\mu}\Big)^{n+3-4\alpha}\Big(\f{\varrho}{\rho}\Big)^{n+3-4\alpha}
\B[E(\varrho)+E_{\ast}(\varrho)\B] E_{q,\ell}(u,\varrho)\nonumber\\&+
\Big(\f{\rho}{\mu}\Big)^{n+3-4\alpha}\Big(\f{\varrho}{\rho}\Big)
E_{q,\ell}(u,\varrho)\nonumber\\
&\times\B[\left(\dfrac{\varrho}{\rho}\right)^{n+2-4\alpha}
E^{\f{1}{q}}(\varrho)E^{1-\f{1}{q}}_{\ast}(\varrho)
+C
\left(\f{\rho}{\varrho}\right)
^{4\alpha-\f{n}{\ell}-1}P_{q',\ell'}(\varrho)\B].\nonumber
 \end{align}
It follows from \eqref{wzp} that
$$\ba
P_{q',\ell'}(\mu)&\leq \left(\dfrac{\varrho}{\mu}\right)^{n+2-4\alpha}
E^{\f{1}{q}}(\varrho)E^{1-\f{1}{q}}_{\ast}(\varrho)
+C
\left(\f{\mu}{\varrho}\right)
^{4\alpha-\f{n}{\ell}-1}P_{q',\ell'}(\varrho).\ea$$
We denote $G_{4}(\mu)=E(\mu)+E_{\ast}(\mu)+\varepsilon_{5}^{1/4}P_{_{q',\ell'}}(\mu)$ and $\lambda=\f{\mu}{\rho}=\f{\rho}{\varrho} $, then the above estimates imply that
$$\ba
G_{4}(\mu)
\leq&
\lambda^{8\alpha-n-4}G_{4}(\varrho)
+\lambda^{8\alpha -2n-6}G_{4}(\varrho)E_{q,\ell}(u,\varrho)\\
&+\lambda^{8\alpha-2n-6}E_{q,\ell}(u,\varrho)G_{4}(\varrho)+
\lambda^{8\alpha-n-5-\f{n}{\ell}}\varepsilon_{5}^{-1/4}G_{4}(\varrho)E_{q,p}(u;\varrho)\\
&+\lambda^{4\alpha-n-2}\varepsilon_{5}^{1/4}G_{4}(\varrho)+
\lambda^{4\alpha-\f{n}{\ell}-1}G_{4}(\varrho)\\
\leq&C \lambda^{\tau_{9}}G_{4}(\varrho)+\lambda^{-\tau_{10}}G_{4}(\varrho)
\varepsilon_{5}^{\tau_{11}},
\ea$$
where $\tau_{9}$, $\tau_{10}$, $\tau_{11}>0.$
Based on the latter relation, along the same line of  the
iteration argument used in the proof   Theorem \ref{the1.2}, one can complete the proof of  second part
of Theorem \ref{the1.4}.
\end{proof}

\appendix
\section{Proof  of useful inequalities}
\label{appendix}
To make our paper more self-contained and more readable,
we   outline the   proof of inequalities \eqref{keyinq1}-\eqref{keyinq8} stated in Section 2.
It is clear that it is enough to show the following inequalities
\begin{align}
\|u-\overline{u}_{\mu}\|_{L^{\f{2n}{n-2s}}(B(\mu/2))}\leq C\B(\int_{B^{\ast}(\mu)}y^{1-2s}|\nabla^{\ast} u^{\ast}|^{2}dxdy\Big)^{1/2},\label{akeyinq1}\\
\int_{B^{\ast}(\mu)}y^{1-2s}|u^{\ast}|^{2}\leq C\mu^{2-2s}
\int_{B(\mu)}|u|^{2}+C\mu^{2}\int_{B^{\ast}(\mu)}
y^{1-2s}|\nabla^{\ast} u^{\ast}|^{2}dxdy,\label{akeyinq3}
\\
\|u\|_{L^{\f{2n}{n-2s}}(B(\mu/2))}\leq C\B(\int_{B^{\ast}(\mu)}y^{1-2s}|\nabla^{\ast} u^{\ast}|^{2}dxdy\Big)^{1/2}+C\mu^{-s}
\B(\int_{B(\mu)}|u|^{2}\B)^{1/2},\label{akeyinq2}\\
\|u\|_{L^{\f{2n}{n-2s}}(B(2/3))}\leq C\B(\int_{B^{\ast}(1)}y^{1-2s}|\nabla^{\ast} u^{\ast}|^{2}dxdy\Big)^{1/2}+C
\B(\int_{B(1)}|u|^{2}\B)^{1/2}.\label{akeyinq6}
\end{align}

\begin{proof}
Consider the usual cut-off functions
$$\eta_{1}(x)=\left\{\ba
&1,\,~~~x\in B(\hbar\mu ),~0<\hbar<1,\\
&0,\,~~~x\in B^{c}(\mu),
\ea\right.
$$
and
$$\eta_{2}(y)=\left\{\ba
&1,\, ~~~0\leq y\leq \hbar\mu,\\
&0,\,~~~ y>\mu,
\ea\right.
$$
 satisfying
$$
0\leq \eta_{1},\,\eta_{2}\leq1,
~~~\text{and}~~~\mu |\partial_{x}\eta_{1}(x)|
+\mu|\partial_{y}\eta_{2}(y)|\leq C.
$$
It is easy to derive from \eqref{eqnorm} and \eqref{mini} that
\begin{align}
\|u\eta_{1}\|^{2}_{\dot{H}^{s}}&=
\int_{\mathbb{R}_{+}^{n+1}} y^{1-2s} | \nabla^{\ast} (u\eta_{1})^{\ast} |^{2}dxdy\nonumber\\
&\leq C\int_{\mathbb{R}_{+}^{n+1}} y^{1-2s} | \nabla^{\ast}   (u^{\ast}\eta_{2}\eta_{1})|^{2}dxdy\nonumber\\
&\leq C\mu^{-2}\int_{B^{\ast}(\mu)}y^{1-2s}|u^{\ast}|^{2}dxdy+
C\int_{B^{\ast}(\mu)}y^{1-2s}|\nabla^{\ast}u^{\ast}|^{2}dxdy.\label{A.1}
\end{align}
Thanks to the classical weighted Poincar\'e inequality, we infer that
\be\label{app2}
\int_{B^{\ast}(\mu)}y^{1-2s}|u^{\ast}- \overline{u^{\ast}}_{B^{\ast}(\mu)}|^{2}\leq C\mu^{2}\int_{B^{\ast}(\mu)}y^{1-2s}|\nabla^{\ast}u^{\ast}|^{2},
\ee
where
$$
\overline{u^{\ast}}_{B^{\ast}(\mu)}=\f{1}{|B^{\ast}(\mu)|} \int_{B^{\ast}(\mu)}y^{1-2s} u^{\ast}dxdy
$$
and the volume of $|B^{\ast}(\mu)|=\int_{B^{\ast}(\mu)}y^{1-2s} dxdy$ is proportional to $\mu^{n+2-2s}$.\\
The above inequalities imply
\be\B(\int_{B(\hbar\mu)}|u-\overline{u^{\ast}}_{B^{\ast}(\mu)}|
^{\f{2n}{n-2s}}\B)^{\f{n-2s}{n}}
\leq C \int_{B^{\ast}(\mu)}y^{1-2s}|\nabla^{\ast}u^{\ast}|^{2}.
\label{app1}\ee
It follows from  $u^{\ast}(x,y)= u(x)+\int^{y}_{0}\partial_{z}u^{\ast}dz$ that
\begin{align}
\bigg|\overline{u^{\ast}}_{B^{\ast}(\mu)}-\overline{u}_{\mu}\bigg|&=\bigg|
\f{1}{|B^{\ast}(\mu)|} \int_{B^{\ast}(\mu)}y^{1-2s}
\int^{y}_{0}\partial_{z}u^{\ast}dz\bigg| \nonumber\\
&\leq \f{1}{|B^{\ast}(\mu)|}\int_{B (\mu)}\int_{0}^{\mu}y^{1-2s}\B(\int^{y}_{0}z^{1-2s}|\partial_{z}u^{\ast}|^{2}dz\B)^{1/2}
\B(\int^{y}_{0}z^{-(1-2s)}dz\B)^{1/2}dxdy \nonumber\\
&\leq C\mu^{s-\f{n}{2}}\B(
\int_{B^{\ast}(\mu)} y^{1-2s}|\nabla^{\ast } u^{\ast}|^{2}dxdy\B)^{1/2},
\label{app3}
\end{align}
where the H\"older inequality was used twice.\\
 Combining
\eqref{app1} with the latter inequality, we deduce that
$$\ba
\B(\int_{B(\hbar\mu)}|u-\overline{u}_{\mu}|^{\f{2n}{n-2s}}\B)^{\f{n-2s}{2n}}
\leq& \B(\int_{B(\hbar\mu)}|u-\overline{u}_{\mu}|^{\f{2n}{n-2s}}\B)^{\f{n-2s}{2n}}\\
\leq &\B(\int_{B(\hbar\mu)}|u-\overline{u^{\ast}}_{B^{\ast}(\mu)}|^{\f{2n}{n-2s}}\B)^{\f{n-2s}{2n}}
\\&+ \B(\int_{B(\hbar\mu)}|\overline{u^{\ast}}_{B^{\ast}(\mu)}-\overline{u}_{\mu}|^{\f{2n}{n-2s}}\B)^{\f{n-2s}{2n}}
\\
\leq& C \B(\int_{B^{\ast}(\mu)}y^{1-2s}|\nabla^{\ast}u^{\ast}|^{2}\B)^{1/2},
\ea$$
which means \eqref{akeyinq1}.

By means of \eqref{app2} and \eqref{app3}, we have
$$\ba
\int_{B^{\ast}(\mu)}y^{1-2s}|u^{\ast}|^{2}\leq& C \int_{B^{\ast}(\mu)}y^{1-2s}|u^{\ast}- \overline{u^{\ast}}_{B^{\ast}(\mu)}|^{2}\\& +C
\int_{B^{\ast}(\mu)}y^{1-2s}|\overline{u^{\ast}}_{B^{\ast}(\mu)} -\overline{u}_{\mu} |^{2}
+C\int_{B^{\ast}(\mu)}y^{1-2s}|\overline{u}_{\mu}|^{2} \\
\leq& C \mu^{2}\int_{B^{\ast}(\mu)}y^{1-2s}|\nabla^{\ast}u^{\ast}|^{2}
+C\mu^{2-2s}\int_{B(\mu)}|u|^{2},
\ea$$
which together with \eqref{A.1}  yields
$$\B(\int_{B(\hbar\mu)}|u|
^{\f{2n}{n-2s}}\B)^{\f{n-2s}{n}}
\leq C \int_{B^{\ast}(\mu)}y^{1-2s}|\nabla^{\ast}u^{\ast}|^{2}+C\mu^{-2s}\int_{B(\mu)}|u|^{2},$$
and in turn implies \eqref{akeyinq3}, \eqref{akeyinq2} and \eqref{akeyinq6}.
\end{proof}

\section*{Acknowledgement}
 The third author is supported in part by the National Natural Science Foundation of China under grant No.11101405.

\end{document}